\documentclass[a4paper,11pt]{amsart}
\usepackage{amsmath}
\usepackage{amsthm}
\usepackage{amsfonts}
\usepackage{amssymb}
\usepackage{latexsym}
\usepackage{mathrsfs}

\usepackage{enumitem}

\usepackage[section]{placeins}
\usepackage{xcolor}

\usepackage[abs]{overpic}
\usepackage[titletoc]{appendix}
\usepackage{graphicx}
\usepackage{latexsym}
\usepackage[ansinew]{inputenc}
\usepackage{epsfig}
\usepackage{psfrag}

\usepackage{tikz}
\usepackage[normalem]{ulem}

\numberwithin{equation}{section}

\newtheorem{theorem}{Theorem}[section]{\bf}{\it}
\newtheorem{lemma}[theorem]{Lemma}{\bf}{\it}
\newtheorem{proposition}[theorem]{Proposition}{\bf}{\it}
\newtheorem{corollary}[theorem]{Corollary}{\bf}{\it}
{\bf}{\it} 
{\bf}{\it}
\newtheorem*{theorem*}{Theorem}

\newtheorem{remark}[theorem]{Remark}
{\bf}{\it}
{\bf}{\it}

\newtheorem*{namedtheorem}{\theoremname}
\newcommand{\theoremname}{testing}
\newenvironment{named}[1]{\renewcommand{\theoremname}{#1}\begin{namedtheorem}}{\end{namedtheorem}}

\newtheorem*{definition*}{Definition}

{\bf}{\it}
{\bf}{\it}
{\bf}{\it}
\newtheorem{example}[theorem]{Example}
\newtheorem*{example*}{Example}

\theoremstyle{remark}

\theoremstyle{definition}

\theoremstyle{remark}

\newcommand{\diam}{\operatorname{diam}}

\newcommand{\bSn}{{\mathbb S^n}}

\newcommand{\R}{\mathbb R}
\newcommand{\Z}{\mathbb Z}
\newcommand{\C}{\mathbb C}
\newcommand{\N}{\mathbb N}

\newcommand{\loc}{{\operatorname{loc}}}

\newcommand{\dist}{{\operatorname{dist}\,}}
\newcommand{\id}{{\operatorname{id}}}

\newdimen\vintkern\vintkern11pt
\def\vint{-\kern-\vintkern\int}

\newcommand{\norm}[1]{\lVert #1 \rVert}

\newcommand{\bS}{\mathbb{S}}

\newcommand{\vol}{\mathrm{vol}}

\newcommand{\ccF}{\mathscr{F}}

\newcommand{\SO}{\mathrm{SO}}
\newcommand{\CO}{\mathrm{CO}}
\newcommand{\Gr}{\mathrm{Gr}}

\newcommand{\HS}{\mathrm{HS}}

\newcommand{\sym}{\mathrm{sym}}

\renewcommand{\le}{\leqslant}
\renewcommand{\ge}{\geqslant}

\title[]{Quasiregular curves of small distortion in product manifolds}

\author{Susanna Heikkil\"a}
\address{Department of Mathematics and Statistics, P.O. Box 68 (Pietari Kalmin katu 5), FI-00014 University of Helsinki, Finland}
\email{susanna.a.heikkila@helsinki.fi}

\author{Pekka Pankka}
\address{Department of Mathematics and Statistics, P.O. Box 68 (Pietari Kalmin katu 5), FI-00014 University of Helsinki, Finland}
\email{pekka.pankka@helsinki.fi}

\author{Eden Prywes}
\address{Department of Mathematics, Princeton University, Princeton, New Jersey 08544}
\email{eprywes@princeton.edu}

\date{\today}

\thanks{S.H. and P.P. were supported in part by the Academy of Finland project \#332671. 
E.P. was partially supported by the NSF grant RTG-DMS-1502424.}
\date{\today}

\begin{document}

\begin{abstract}
We consider, for $n\ge 3$, $K$-quasiregular $\vol_N^\times$-curves $M\to N$ of small distortion $K\ge 1$ from oriented Riemannian $n$-manifolds into Riemannian product manifolds $N=N_1\times \cdots \times N_k$, where each $N_i$ is an oriented Riemannian $n$-manifold and the calibration $\vol_N^\times\in \Omega^n(N)$ is the sum of the Riemannian volume forms $\vol_{N_i}$ of the factors $N_i$ of $N$. 

We show that, in this setting, $K$-quasiregular curves of small distortion are carried by quasiregular maps. More precisely, there exists $K_0=K_0(n,k)>1$ having the property that, for $1\le K\le K_0$ and a $K$-quasiregular $\vol_N^\times$-curve $F=(f_1,\ldots, f_k) \colon M \to N_1\times \cdots \times N_k$ there exists an index $i_0\in \{1,\ldots, k\}$ for which the coordinate map $f_{i_0}\colon M\to N_{i_0}$ is a quasiregular map. As a corollary, we obtain first examples of decomposable calibrations for which corresponding quasiregular curves of small distortion are discrete and admit a version of Liouville's theorem.
\end{abstract}

\maketitle



\section{Introduction}

In this article, we consider the quasiconformal geometry of mappings $M\to N$ between Riemannian manifolds, where $\dim M \le \dim N$. Our main focus is in the stability of Liouville's theorem and in Reshetnyak's theorem for quasiregular curves of small distortion. To set the stage, we discuss first Liouville's theorem for conformal curves associated to calibrations and then discuss results related to curves of small distortion. Here we call a mapping, associated to a calibration of manifold, a \emph{curve}. We begin by introducing this terminology in more detail.

Let $M$ and $N$ be Riemannian manifolds of dimensions $n=\dim M\le \dim N$. An $n$-form $\omega\in \Omega^n(N)$ is a \emph{calibration on $N$} if $\omega$ is a closed non-vanishing $n$-form having point-wise comass norm $\norm{\omega}$ equal to one. 
Recall that the comass norm of an $n$-form $\omega \in \Omega^n(N)$ at $p\in N$ is
\[
\norm{\omega_p} = \max\{ \omega_p(v_1,\ldots, v_n) \colon |v_1|\le 1, \ldots, |v_n|\le 1\}.
\]
For example, the Riemannian volume form of a Riemannian manifold and the symplectic form of a K\"ahler manifold are calibrations. We refer to the seminal article of Harvey and Lawson \cite{Harvey-Lawson-Acta-1982} for calibrations associated to exceptional Riemannian geometries and for the role of calibrations in the theory of minimal surfaces.

We say that a continuous mapping $F\colon M\to N$ in $W^{1,n}_\loc(M,N)$, where $n=\dim M$, is an \emph{$\omega$-calibrated curve for a calibration $\omega\in \Omega^n(N)$ on $N$} if 
\begin{equation}
\label{eq:associated}
\norm{DF}^n = \star F^*\omega
\end{equation}
almost everywhere in $M$, where $\norm{DF}$ is the operator norm of the differential and $\star$ is the Hodge star operator. For any calibration, a mapping satisfying \eqref{eq:associated} is weakly conformal; see Lemma \ref{lemma-linearmaprigidity}. Here, and in what follows, a continuous mapping $F\colon M\to N$ between Riemannian manifolds is \emph{weakly conformal} if $F$ belongs to the Sobolev space $W^{1,n}_\loc(M,N)$ for $n=\dim M$ and the differential $DF$ of $F$ is a conformal linear map almost everywhere in $M$. 

Conformal mappings and holomorphic curves are well-understood examples of calibrated curves.

\begin{example}
A continuous map $M\to N$ between Riemannian manifolds of the same dimension is a calibrated curve if and only if it is a conformal map, that is, a $C^\infty$-smooth local homeomorphism having conformal non-vanishing differential. Here the smoothness follows from Ferrand's theorem \cite{Lelong-Ferrand-MPAM-1976} on smoothness of conformal mappings between smooth manifolds. 

In particular, by the classical Liouville's theorem, a continuous map $\Omega \to \bS^n$, where $\Omega$ is a domain in $\bS^n$, is a calibrated curve if and only if it is a restriction of a M\"obius transformations by the classical Liouville's theorem. We refer to Iwaniec and Martin \cite{Iwaniec-Martin-Acta-1993} and \cite[Section 5]{Iwaniec-Martin-book}, Shachar \cite{Shachar-arXiv-2020}, and Liu \cite{Liu-Advances-2013} for the discussion of Liouville's theorem under low regularity assumptions.
\end{example}
  
\begin{example}
By elementary linear algebra (see Remark \ref{rmk:example-holomorphic}), curves $\Omega \to \C^k$ calibrated by the standard symplectic form $\omega_\sym = dx_1\wedge dy_1 + \cdots + dx_k \wedge dy_k \in \Omega^2(\C^k)$, where $\Omega \subset \C$ is a domain, have weakly conformal coordinate functions. 
Thus, by Weyl's lemma, a map $\Omega \to \C^k$ is an $\omega_\sym$-calibrated curve if and only if it is a holomorphic curve. 

More generally, a continuous map $\Sigma \to N$ from a Riemann surface to a K\"ahler manifold $(N,J,\omega)$ is an $\omega$-calibrated curve if and only if it is a $J$-holomorphic curve. The smoothness of the curve follows from higher integrability and elliptic regularity; see e.g.~McDuff and Salamon \cite[Theorem B.4.1]{McDuff-Salamon-book} and the discussion below. We refer to Gromov \cite{Gromov-Inventiones-1985} for more details on $J$-holomorphic curves.
\end{example}

\begin{example}
An almost complex structure $J$ on a K\"ahler manifold $N$ is a special case of a vector cross product $Q\colon \wedge^{k-1} TN \to TN$. To each vector cross product $Q$, we may associate a calibration $\omega_Q \in \Omega^k(N)$ by 
\[
\omega_Q(v_1,\ldots, v_k) = \langle Q(v_1,\ldots, v_{k-1}),v_k\rangle,
\]
where $v_1,\ldots, v_k\in T_pN$ for $p\in N$.

A continuous map $F\colon M \to N$ in $W^{1,n}_\loc(M,N)$, where $n=\dim M$, is a \emph{Smith map with respect to an $(n-1)$-vector cross product $Q$ on $N$} if 
\[
Q\circ \wedge^{n-1} DF = \norm{DF}^{n-2} DF \circ \star
\]
almost everywhere on $M$. By Cheng, Karigiannis, and Madnick \cite[Propositions 2.32 and 3.2]{Cheng-Karigiannis-Madnick-arXiv-2019}, a continuous Sobolev map $F\colon M\to N$ in $W^{1,n}_\loc(M,N)$ is a Smith map with respect to a vector cross product $Q$ if and only if $F$ is a $\omega_Q$-calibrated curve with respect to the calibration $\omega_Q$ associated to $Q$.

We refer to \cite{Cheng-Karigiannis-Madnick-arXiv-2019} for further discussion on the Brown--Gray classification \cite{Brown-Gray-Commentarii-1967} of vector cross products and properties of Smith maps.
\end{example}

\begin{example}
In the previous examples, curves stem from the point of view of classical mapping theory. To the other direction, we may interpret $\omega$-calibrated curves as holonomic sections of partial differential relations associated to the calibration $\omega$; we refer to Gromov's book \cite{Gromov-PDR-book} for the terminology. Indeed, given a calibration $\omega\in \Omega^n(N)$, let $G(\omega,p) \subset \widetilde{\Gr}(n,T_pN)$ be the maximal $n$-planes of the calibration $\omega_p$ of $T_pN$ in the oriented Grassmannian $\widetilde{\Gr}(n,T_pN)$. Then $F\colon M\to N$ is an $\omega$-calibrated curve if and only if $F$ is a continuous Sobolev map in $W^{1,n}_\loc(M,N)$ satisfying 
\begin{equation}
\label{eq:PDR}
(DF)_x \in \{ L \in \CO(T_x M, T_{F(x)} N) \colon  L(T_xM) \in G(\omega,F(x))\}
\end{equation}
for almost every $x\in M$, where $\CO(T_x M, T_{F(x)} N)$ is the space of conformal linear maps from $T_x M$ to $T_{F(x)}N$ and $T_xM$ is oriented by $\vol_M$. 
\end{example}

For the volume form $\vol_N$ or the symplectic form $\omega_\sym$, all solutions to the partial differential relation \eqref{eq:PDR} are smooth. This follows from the the standard elliptic regularity theory for the Laplacian. By the results of Cheng, Karigiannis, and Madnick in  \cite{Cheng-Karigiannis-Madnick-arXiv-2019}, Smith maps associated to vector cross products are $C^{1,\alpha}$-regular and $C^\infty$-smooth away from the critical set. This follows from the regularity theory of $n$-harmonic mappings. In fact, all calibrated curves are energy minimizers for the $n$-energy; see Section \ref{sec:n-harmonic}. 

To our knowledge, it is an open question whether all solutions $M\to N$ of the partial differential relation \eqref{eq:PDR} are $C^\infty$-smooth. 
Also, to our knowledge, it is an open question whether the critical set of a non-constant $\omega$-calibrated curve $M\to N$ is empty. Weakly conformal maps in general are too flexible to have such additional regularity. Consider, for example, a folding map $\R^n \to \R^n$, $(x_1,\ldots, x_n) \mapsto (|x_1|,\ldots, |x_n|)$, which is a Lipschitz regular weakly conformal map.

\subsection{Liouville theorems for curves into product manifolds}

Our first main theorem is a Liouville theorem for calibrated curves $M \to N$ into product manifolds $N=N_1\times \cdots \times N_k$, where the calibration is a canonical multisymplectic form on $N$. More precisely, let $N=N_1\times \cdots \times N_k$ be a Riemannian product manifold of dimension $nk$, where each factor $N_i$ is $n$-dimensional. The manifold $N$ carries a canonical $n$-form 
\[
\vol_N^\times = \sum_{i=1}^k \pi_i^* \vol_{N_i} \in \Omega^n(N)
\]
associated to the product structure $N_1\times \cdots \times N_k$ of $N$, where each $\vol_{N_i}$ is the Riemannian volume form of the manifold $N_i$ and each $\pi_i \colon N \to N_i$ is a projection $(p_1,\ldots, p_k) \mapsto p_i$. Note that $\vol_{\C^k}^\times$ is the classical symplectic form $\omega_\sym$ on $\C^k$. For detailed discussion on multisymplectic forms see e.g.~Cantrijn, Ibort, and de Le\'on \cite{Cantrijn-Ibort-Leon-JAMSSA-1999}.

In contrast to holomorphic curves, $\vol_N^\times$-calibrated curves in dimensions $n\ge 3$ are extremely rigid. Heuristically, we may say that these curves are carried by conformal maps in the sense of the following theorem.

\begin{theorem}
\label{thm:manifold-Liouville} 
Let $M$ be an oriented and connected Riemannian $n$-manifold for $n\ge 3$ and let $N=N_1\times \cdots \times N_k$ be a Riemannian product of oriented Riemannian $n$-manifolds $N_i$ for $i\in \{1,\ldots, k\}$. Then for each non-constant $\vol_N^\times$-calibrated curve $F=(f_1,\ldots, f_k)\colon M\to N$ there exists an index $i_0\in \{1,\ldots, k\}$ for which
\begin{enumerate}
\item the coordinate map $f_{i_0}\colon M\to N_{i_0}$ is a conformal map and
\item for $i\ne i_0$ the coordinate map $f_i \colon M\to N_i$ is constant.
\end{enumerate}
\end{theorem}

Note that, by Ferrand's theorem, $f_{i_0}$ and hence also $F$ are $C^\infty$-smooth. In the particular case that $M$ is a domain in $\bS^n$ and $N = (\bS^n)^k$, we further have by Liouville's theorem that $f_{i_0}$ is a restriction of a M\"obius transformation.

We may interpret Theorem \ref{thm:manifold-Liouville} also in terms of minimal surfaces. Since the bundle $G(\vol_N^\times) \to N$, where $G(\vol_N^\times) = \bigcup_{p\in N} G(\vol_N^\times,p)$, is a product bundle with discrete fibers, it is easy to see that a $C^1$-smooth minimal surface is contained in a submanifold, which is an isometric copy of a factor $N_i$ in $N$. Theorem \ref{thm:manifold-Liouville} now shows that image of an $\vol_N^\times$-calibrated curve is always contained in such a minimal surface.

\subsection{From conformal to quasiconformal geometry}

Theorem \ref{thm:manifold-Liouville} shows that, for $n\ge 3$, calibrated $\vol_N^\times$-curves $M\to N$ exhibit the same rigidity as the conformal mappings between manifolds of the same dimension.  To prove Theorem \ref{thm:manifold-Liouville} we consider curves of small distortion.

A continuous mapping $F\colon M\to N$ between Riemannian manifolds $M$ and $N$ of dimensions $n=\dim M\le \dim N$ is a \emph{$K$-quasiregular $\omega$-curve}, for $K\ge 1$ and a non-vanishing closed form $\omega\in \Omega^n(N)$, if $F$ belongs to the Sobolev space $W^{1,n}_\loc(M,N)$ and satisfies the distortion inequality
\begin{equation}
\label{eq:QRC}
(\norm{\omega}\circ F) \norm{DF}^n \le K (\star F^*\omega)
\end{equation}
almost everywhere in $M$. If $\omega$ is a calibration and $K=1$, we recover Equation \eqref{eq:associated}.

The term quasiregular curve stems from the notion of quasiregular mapping. A continuous mapping $F\colon M\to N$ between oriented Riemannian $n$-manifolds is \emph{$K$-quasiregular for $K\ge 1$} if $F\in W^{1,n}_\loc(M,N)$ and
\[
\norm{DF}^n \le K J_F
\]
almost everywhere in $M$, where $J_F$ is the Jacobian determinant of $F$. Since $J_F = \star F^*\vol_N$, we readily observe that a continuous mapping $M\to N$ in the Sobolev space $W^{1,n}_\loc(M,N)$ is a $K$-quasiregular mapping if and only if it is a $K$-quasiregular $\omega$-curve for a non-vanishing form $\omega \in \Omega^n(N)$.

We refer to monographs of Iwaniec and Martin \cite{Iwaniec-Martin-book}, Reshetnyak \cite{Reshetnyak-book}, and Rickman \cite{Rickman-book} for the theory of quasiregular mappings and to \cite{Heikkila-arXiv-2021},  \cite{Onninen-Pankka-arXiv-2020}, and \cite{Pankka-AASF-2020} for discussions on properties of quasiregular curves. We merely note that, similarly as for quasiregular mappings, the defining condition for quasiregular curves is a $C^0$-closed partial differential relation by \cite[Theorem 1.9]{Pankka-AASF-2020}; see Gromov \cite[Section 1.2.3]{Gromov-PDR-book} for the terminology.

As our first result on quasiregular curves, we show that quasiregular $\vol_N^\times$-curves $M\to N$ of small distortion are local quasiconformal embeddings. 

\begin{theorem}
\label{thm:local-embedding}
Let $n\ge 3$, $k\ge 1$, and $H>1$. Then there exists $\varepsilon=\varepsilon(n,k,H)>0$ having the property that each non-constant $(1+\varepsilon)$-quasiregular $\vol_N^\times$-curve $M\to N$ from an oriented and connected Riemannian $n$-manifold $M$ to a Riemannian product $N=N_1\times \cdots \times N_k$ of oriented Riemannian $n$-manifolds is a local $H$-quasiconformal embedding.
\end{theorem}

Since the analytic definition of quasiconformal mappings is not at our disposal in this context, Theorem \ref{thm:local-embedding} is stated in terms of the metric definition of quasiconformal mappings. A local embedding $F\colon M\to N$ between Riemannian manifolds is \emph{$H$-quasiconformal for $H\ge 1$}, if
\[
\limsup_{r\to 0} \frac{\sup_{d_M(y,x)=r} d_N(F(y),F(x))}{\inf_{d_M(y,x)=r} d_N(F(y),F(x))} \le H
\]
for every $x\in M$. We refer to Heinonen and Koskela \cite{Heinonen-Koskela-Acta-1998} for the metric theory of quasiconformal mappings.

Our interest to Theorem \ref{thm:local-embedding} stems from the classical Reshetnyak's theorem for quasiregular mappings \cite{Reshetnyak-Sibirsk-1967}: \emph{a non-constant quasiregular mapping between oriented Riemannian $n$-manifolds is discrete and open}. Although quasiregular curves into higher dimensional targets are never open, Theorem \ref{thm:local-embedding} shows that $\vol_N^\times$-quasiregular curves $M\to N$ of small distortion are discrete in the case $\dim M \ge 3$. It should be noted that this discreteness property fails for quasiregular $\vol_{(\R^n)^k}^\times$-curves of large distortion and even for quasiregular $\omega_\sym$-curves of small distortion.

For curves of large distortion, the seminal example is by Iwaniec, Verchota, and Vogel \cite{Iwaniec-Verchota-Vogel-ARMA-2002}: \emph{There exists a Lipschitz mapping $F=(f_1,f_2) \colon \C\to \C^2$ which satisfies $J_{f_1}+J_{f_2} =1$ in the upper half-plane and is constant in the lower half-plane.} To our knowledge, this map constructed by Iwaniec, Verchota, and Vogel is the first example of a non-constant quasiregular $\omega_\sym$-curve, which is not discrete. We show in Section \ref{sec:IVV} that, for each $n\ge 3$, there exists $k\in \N$ and a non-constant quasiregular $\vol_{(\R^n)^k}^\times$-curve $F \colon \R^n \to (\R^n)^k$ which is constant in the lower half-space $\R^n_-$. In particular, these non-constant curves are not discrete.

The failure of discreteness for quasiregular $\omega_\sym$-curves of small distortion follows from an example of Rosay \cite{Rosay-ANF-2010}: \emph{For every $K>1$, there exist a non-constant map $F \colon \C\to \C^2$ satisfying
\begin{equation}
\label{eq:partial-bar}
\left| \partial_{\bar z} F \right| \le \frac{K-1}{K+1} \left| \partial_z F \right|,
\end{equation}
which is not discrete.} In particular, Rosay's example shows that Beltrami systems $\partial_z F = \mu \partial_z F$ in $\C^2$ carry non-discrete solutions independently of the size of the Beltrami coefficient $\mu$. Rosay's map is also a $K$-quasiregular $\omega_\sym$-curve. We discuss this example in more detail in Section \ref{sec:Rosay}.

\subsection{Stability of Liouville theorems}

We return now to an analog of Theorem \ref{thm:manifold-Liouville} for quasiregular $\vol_N^\times$-curves $M\to N$ of small distortion. 
We show that, heuristically, these curves do not change their direction in the Grassmannian $\widetilde{\Gr}(n,TN)$. 

\begin{theorem}
\label{thm:stability-main}
Let $n\ge 3$ and $k\ge 1$. Then there exists $\varepsilon=\varepsilon(n,k)>0$ for the following. Let $M$ be an oriented and connected Riemannian $n$-manifold and let $N=N_1\times \cdots \times N_k$ be a Riemannian product of oriented Riemannian $n$-manifolds. Then, for a non-constant $(1+\varepsilon)$-quasiregular $\vol_N^\times$-curve $F=(f_1,\ldots, f_k) \colon M\to N$, there exists a unique index $i_0\in \{1,\ldots, k\}$ for which the coordinate map $f_{i_0} \colon M \to N_{i_0}$ is a quasiregular local homeomorphism.
\end{theorem}

Theorem \ref{thm:stability-main} resembles a theorem of Ball and James \cite{Ball-James-ARMA-1987} on the convergence of differentials of Sobolev maps. Indeed, by Theorem \ref{thm:manifold-Liouville}, $\vol_{(\R^n)^k}^\times$-calibrated curves $\Omega \to (\R^n)^k$ have rigid derivatives. Thus we may view Theorem \ref{thm:stability-main} as a convergence theorem of differentials as the distortion of the quasiregular $\vol_{(\R^n)^k}^\times$-curves tends to one. See also Kirchheim and Sz\'ekelyhidi \cite{Kirchheim-Szekelyhidi-Crelle-2008} on related results on the rank one convex hull of differentials of Lipschitz maps.

\begin{remark}
The smallness of the distortion is crucial in Theorem \ref{thm:stability-main}. The curve constructed in Section \ref{sec:IVV} gives also an example of a quasiregular $\vol_{(\R^n)^k}^\times$-curve $\R^n \to (\R^n)^k$ without quasiregular component mappings.
\end{remark}

We record two corollaries of Theorem \ref{thm:stability-main}. First we observe that these curves have the same local properties as quasiregular mappings; see e.g.~Rickman \cite[Chapter I]{Rickman-book}.

\begin{corollary}
Let $n\ge 3$ and $k\ge 1$. Then there exists $\varepsilon_0=\varepsilon_0(n,k)>0$ for the following. Let $F\colon M\to N$ be a non-constant $(1+\varepsilon_0)$-quasiregular map from an oriented Riemannian $n$-manifold $M$ to a product $N=N_1\times \cdots \times N_k$ of oriented Riemannian $n$-manifolds. Then
\[
\star F^*\vol_N^\times > 0
\]
almost everywhere in $M$ and the set
\[
B_F = \{ x\in M \colon F \text{ is not a local embedding at } x\}
\]
has topological dimension at most $n-2$ and Hausdorff $n$-measure zero.
\end{corollary}

Second, we observe that, if $M$ is conformally equivalent to a complete manifold with finite volume and $N$ is simply connected, we have, by the Gromov--Zorich global homeomorphism theorem for quasiregular mappings, that the coordinate mapping $f_{i_0}\colon M\to N_{i_0}$ carrying the curve $F\colon M\to N$ is an embedding with a dense image in $N_{i_0}$; see Zorich  \cite{Zorich-MatSb-1967} for the Euclidean global homeomorphism theorem and Gromov \cite[Remark 6.30]{Gromov-book} and Zorich \cite{Zorich-FAP-2000} for the Riemannian case, or \cite{Holopainen-Pankka-AASF-2004} for the Gromov--Zorich theorem for mappings of finite distortion between Riemannian manifolds. In particular, we have the following corollary.

\begin{corollary}
\label{cor:stability-qr-curves}
Let $n\ge 3$ and $k\ge 1$. Then there exists $\varepsilon_0=\varepsilon_0(n,k)>0$ having the property that, if there exists a non-constant $(1+\varepsilon_0)$-quasiregular $\vol_N^\times$-curve $\R^n \to N$ for $N=N_1\times \cdots \times N_k$, then one of the factors $N_i$ of $N$ is either covered by $\R^n$ or $\bS^n$.
\end{corollary}

To our knowledge, it is an open question whether quasiregular curves $\R^n\to N$ of arbitrary distortion have a similar property, that is, whether one of the factors $N_i$ of $N=N_1\times \cdots \times N_k$ admits a non-constant quasiregular map from $\R^n$ if there exists a non-constant quasiregular $\vol_N^\times$-curve $ \R^n \to N$. Such manifolds $N_i$ are called quasiregularly elliptic. We refer to Bonk and Heinonen \cite{Bonk-Heinonen-Acta-2001} and \cite{Prywes-Annals-2019} for discussion on quasiregularly elliptic manifolds. Note that the problem is not local in the sense that it is easy to find quasiregular $\vol_{(\R^n)^k}^\times $-curves $F=(f_1,\ldots, f_k) \colon \R^n \to (\R^n)^k$ for which none of the coordinate mappings $f_i$ are quasiregular; see the example in Section \ref{sec:IVV}.

\subsection*{Outline of the proofs}

Since the methods used in the proofs of the main theorems may have independent interest, we give an outline leading to the proofs of the main theorems. 

The main issue in the proofs of Theorems \ref{thm:manifold-Liouville} and \ref{thm:stability-qr-curves} is that the differential $DF$ of the curve $F\colon M\to N$ is merely an $L^n$-function. Therefore, although the maximal planes $G(\vol_N^\times ,{F(x)})$ of the calibration $\vol_N^\times$ form a discrete set in the oriented Grassmannian $\widetilde{\Gr}(n,T_{F(x)}N)$ for each $x\in M$, we do not have a priori control on the oscillation of the differential between the maximal planes. The example in Section \ref{sec:IVV} shows that for curves of large distortion this oscillation is, in fact, possible. 

To obtain additional a priori regularity, we begin by showing that calibrated curves $M\to \R^m$ are $C^{1,\alpha}$-regular; see Theorem \ref{thm:Hardt-Lin}. This follows essentially from the local quasiminimality of quasiregular curves (\cite[Theorem 1.6]{Pankka-AASF-2020}). After this a priori regularity result, we work with quasiregular $\vol_{(\R^n)^k}$-curves $\Omega \to (\R^n)^k$ of small distortion, where $\Omega\subset \R^n$ is a domain, and prove first Euclidean versions of the results stated in the introduction.

Using the H\"older continuity of the derivative, we prove in Section \ref{sec:n-harmonic} a Euclidean version (Theorem \ref{thm:Euclidean-Liouville}) of the Liouville theorem (Theorem \ref{thm:manifold-Liouville}). Having this Euclidean version at our disposal, we show, using a method of Iwaniec \cite{Iwaniec-PAMS-1987} for quasiregular mappings, that $\vol_{(\R^n)^k}^\times$-curves $\Omega \to (\R^n)^k$ of small distortion are uniformly close to $\vol_{(\R^n)^k}^\times$-calibrated curves; see Proposition \ref{prop:1-qr-near} for a precise statement. In the course of the proof, we show that bounded $\vol_{(\R^n)^k}^\times$-curves form a normal family (Theorem \ref{thm:normal}).

From Proposition \ref{prop:1-qr-near} we conclude that quasiregular $\vol_{(\R^n)^k}^\times$-curves of small distortion are local quasiconformal embeddings (Theorem \ref{thm:local-injectivity}). From this Euclidean result, Theorem \ref{thm:local-embedding} follows by a simple covering argument with bilipschitz charts.

We use the Euclidean local embedding theorem (Theorem \ref{thm:local-injectivity}) to prove the following version of Kopylov's piecewise linear approximation theorem \cite{Kopylov-1972} for quasiregular $\vol_{(\R^n)^k}^\times$-curves. 

\begin{theorem}
\label{thm:4}
Let $n\ge 3$, $k\ge 1$, and $K\ge 1$.  Then there exists $\varepsilon=\varepsilon(n,k,K)>0$ for the following. For $\delta>0$, a domain $\Omega \subset \R^n$, a compactly contained subdomain $U\Subset \Omega$, and a $(1+\varepsilon)$-quasiregular $\vol_{(\R^n)^k}^\times$-curve $F\colon \Omega\to (\R^n)^k$, there exists a piecewise linear $K$-quasiregular $\vol_{(\R^n)^k}^\times$-curve $\widehat F\colon U \to (\R^n)^k$ satisfying $\norm{F-\widehat F}_\infty < \delta$.
\end{theorem}

Having this piecewise linear approximation theorem at our disposal, we are ready to prove a Euclidean version of Theorem \ref{thm:stability-qr-curves}. The finiteness of the set $G(\vol_{(\R^n)^k}^\times)$ of maximal planes of $\vol_{(\R^n)^k}^\times$ together with Theorem \ref{thm:4} now yields that quasiregular $\vol_{(\R^n)^k}^\times$-curves of small distortion are carried by quasiregular coordinate mappings. 
From this Euclidean result, we obtain a quantitative version of Theorem \ref{thm:stability-main}, which in turn yields Theorem \ref{thm:manifold-Liouville}.



\section{Failure of discreteness of quasiregular $\vol_{(\R^n)^k}^\times$-curves of large distortion}
\label{sec:IVV}

In this section, we construct, for each $n \ge 3$, an example of a non-constant $\vol_{(\R^n)^k}^\times$-curve which is constant in the lower half-space. 

\begin{theorem}
\label{thm:IVV}
For each $n\ge 3$, there exists $k\in \N$ and a non-constant quasiregular $\vol_{(\R^n)^k}^\times$-curve $F\colon \R^n \to (\R^n)^k$ which is constant in the lower half-space $\R^{n-1}\times (-\infty,0]$.
\end{theorem}

Given $n\ge 3$, we fix an orientation preserving $L$-Lipschitz branched cover $A\colon \R^{n-1}\to \bS^{n-1}$ satisfying $J_A\ge (1/L)^{n-1}$ almost everywhere, where $L=L(n)\ge 1$; see e.g.~\cite{Drasin-Pankka} for more details on the construction of $A$ as part of the construction of the Zorich map $Z\colon \R^n \to \R^n$, $(x,t)\mapsto e^t A(x)$. Note that $A$ is, in fact, quasiregular; see Martio and V\"ais\"al\"a \cite{Martio-Vaisala-MathAnn-1988}.

We may now formulate the crux of the proof of Theorem \ref{thm:IVV} as a lemma. Here, and in what follows, we say that a map $G\colon \overline{\Omega}\to N$ into a manifold $N$ is a $K$-quasiregular $\omega$-curve if $G|_{\Omega}\colon \Omega\to N$ is a $K$-quasiregular $\omega$-curve.

\begin{lemma}
\label{lemma:IVV}
Let $n\ge 3$. Then there exists $k\in \N$ and a Lipschitz map $G=(g_1,\ldots, g_k)\colon \R^{n-1}\times [1/2,1]\to (\R^n)^k$, which is a quasiregular $\vol_{(\R^n)^k}^\times$-curve, satisfying $|g_i(x,t)|\le 1$,  $g_i(x,1)=A(x)$ and $g_i(x,1/2) = (1/2)^k A(2x)$ for every $x\in \R^{n-1}$, $t\in [1/2,1]$, and $i\in \{1,\ldots, k\}$.
\end{lemma}

\begin{proof}[Proof of Theorem \ref{thm:IVV} assuming Lemma \ref{lemma:IVV}]
Let $n\ge 3$ and let $k\in \N$ and  $G \colon \R^{n-1}\times [1/2,1]\to (\R^n)^k$ be as in Lemma \ref{lemma:IVV}. Let $F=(f_1,\ldots, f_k)\colon \R^n \to (\R^n)^k$ be the map
\[
(x,t) \mapsto \left\{ \begin{array}{rl}
\left( \frac{1}{2^k}\right)^{\ell-1} G(2^{\ell-1} x,2^{\ell-1} t), & t \in [(1/2)^{\ell}, (1/2)^{\ell-1}],\ \ell\in \Z, \\
0, & t \le 0.
\end{array}\right.
\]

The map $F$ is continuous between the strips by the boundary properties of $g_i$ in Lemma \ref{lemma:IVV}.  The map $F$ is continuous at $t=0$ since $A$ and hence $G$ is uniformly bounded by $1$ and so as $t$ tends to $0$, $|F|$ is bounded by $2\sqrt{k}t$.

Let $\zeta \colon \R^{n-1}\times [0,\infty) \to \R^{n-1}\times [0,\infty)$ be the map $(x,t)\mapsto (2^{\ell-1} x, 2^{\ell-1} t)$ and $\sigma \colon (\R^n)^k \to (\R^n)^k$ be the map $y\mapsto y/(2^k)^{\ell-1}$. Since $\zeta$ and $\sigma$ are both conformal and 
\[
\sigma^*\vol_{(\R^n)^k}^\times = (2^{-(\ell-1)k})^n\vol_{(\R^n)^k}^\times = \norm{D\sigma}^n \vol_{(\R^n)^k}^\times,
\]
we have that 
\[
\norm{DF}^n = \norm{D\sigma}^n \norm{DG}^n \norm{D\zeta}^n
\]
and that
\[
\star F^*\vol_{(\R^n)^k}^\times 
= \zeta^* G^* \sigma^* \vol_{(\R^n)^k}^\times 
= \norm{D\sigma}^n \norm{D\zeta}^n G^* \vol_{(\R^n)^k}^\times.
\]
Thus $F$ is a $K$-quasiregular $\vol_{(\R^n)^k}^\times$-curve.

Since $F$ is constant in the lower half-space $\R^n_-$, the map $F$ is the desired curve.
\end{proof}

\begin{proof}[Proof of Lemma \ref{lemma:IVV}]
We define first an auxiliary mapping $s\colon \R^{n-1}\times [0,1] \to \R^n$ by $(x,t) \mapsto t A(x)$. Then $s$ is a $2L$-Lipschitz mapping satisfying
\[
\frac{\partial s}{\partial x_i}(x,t) = t \frac{\partial A}{\partial x_i}(x) \quad \text{and} \quad
\frac{\partial s}{\partial t}(x,t) = A(x)
\]
almost everywhere in $\R^{n-1}\times [0,1]$. Since
\[
\left\langle  \frac{\partial A}{\partial x_i}(x), A(x) \right\rangle = 0
\]
for almost every $x\in \R^{n-1}$, we have that
\begin{align*}
\star (s^*\vol_{\R^n})_{(x,t)} &= \vol_{\R^n}\left( \frac{\partial s}{\partial x_1}(x,t), \ldots, \frac{\partial s}{\partial x_{n-1}}(x,t), \frac{\partial s}{\partial t}(x,t)\right) \\
&= \vol_{\R^n} \left( t \frac{\partial A}{\partial x_1}(x), \ldots, t \frac{\partial A}{\partial x_{n-1}}(x),  A(x)\right) \\
&= t^{n-1} \left( \star A^*\vol_{\bS^{n-1}} \right)_x \ge \left( \frac{t}{L}\right)^{n-1} \ge \left( \frac{1}{2L}\right)^{n-1}
\end{align*}
for almost every $(x,t) \in \R^{n-1}\times [1/2,1]$.

Let now $h \colon \R^{n-1}\times [1/2,1]\to \R^n$ be the mapping
\[
(x,t)\mapsto \left(\frac{t-1/2}{1-1/2}\right)2A(x) + \left(1- \frac{t-1/2}{1-1/2}\right)A(2x).
\]
Then $h$ is $10L$-Lipschitz and, in particular, 
\[
\star h^*\vol_{\R^n} \ge -(10L)^n.
\]

Let now $k\in \N$ be an integer for which
\begin{equation}
\label{eq:choice-of-k}
(k-1) \left( \frac{1}{2L}\right)^{n-1} - (10L)^n \ge 1. 
\end{equation}

We define now a mapping $H=(h_1,\ldots, h_k)\colon \R^{n-1} \times [(1/2)^k,1] \to (\R^n)^k$ componentwise. For each $i\in \{1,\ldots, k\}$, let $h_i \colon \R^{n-1} \times [(1/2)^k,1] \to \R^n$ be the mapping defined in $\R^{n-1}\times  [(1/2)^\ell, (1/2)^{\ell-1}]$ for $\ell\in \{1,\ldots, k\}$ by the formula
\[
h_i(x,t) = \left\{ \begin{array}{rl}
s(2x,t), & \text{for }i < \ell, \\
\frac{1}{2^\ell} h(x, 2^{\ell-1} t), & \text{for } i = \ell, \\
s(x,t), & \text{for } i>\ell,
\end{array}\right.
\]
where $(x,t)\in \R^{n-1}\times  [(1/2)^\ell, (1/2)^{\ell-1}]$. 
Note that, for every $x\in \R^{n-1}$, we have that 
\[
s(2x,2^{-\ell}) = \frac{1}{2^\ell}A(2x) = \frac{1}{2^\ell} h(x,1/2) = h_\ell(x,2^{-\ell}).
\]
and that
\[
h_\ell(x,2^{-(\ell-1)}) = \frac{1}{2^\ell}h(x,1) = \frac{1}{2^\ell}\cdot 2 A(x) = s(x,2^{-(\ell-1)}).
\]
Thus each $h_i$ is continuous. Furthermore, we have that 
\[
h_i(x,1)=A(x) \quad \text{and} \quad h_i(x,1/2^k) = \frac{1}{2^k} A(2 x).
\]

Since the coordinate mappings of $H$ are $10L$-Lipschitz, the mapping $H$ is $10L\sqrt{k}$-Lipschitz. By \eqref{eq:choice-of-k}, we have in the strip $\R^{n-1}\times [1/2^\ell, 1/2^{\ell-1}]$ that
\begin{align*}
\star H^*\vol_{(\R^n)^k}^\times = \sum_{i=1}^k \star h_i^* \vol_{\R^n} 
&=(k-1) (\star s^*\vol_{\R^n}) + (\star h_\ell^*\vol_{\R^n}) \\
&\ge (k-1) (\star s^*\vol_{\R^n}) - \frac{2^{\ell-1}}{(2^\ell)^n} (10L)^n\\
&\ge (k-1)\left( \frac{1}{2L}\right)^{n-1} - (10L)^n \ge 1
\end{align*}
almost everywhere in $\R^{n-1}\times [(1/2)^k, 1]$. Thus
\[
\norm{DH}^n \le (10L\sqrt{k})^n \left( \star H^*\vol_{(\R^n)^k}^\times\right)
\]
almost everywhere in $\R^{n-1}\times [(1/2)^k,1]$.

Let now $\xi \colon \R^{n-1}\times [1/2, 1]\to \R^{n-1}\times [(1/2)^k,1]$ be a Lipschitz map defined by 
\[
(x,t)\mapsto \left(x,\frac{t-1/2}{1-1/2} + \left( 1-\frac{t-1/2}{1-1/2} \right) (1/2^k)\right).
\]

We define now  $G = (g_1,\ldots, g_k) \colon \R^{n-1}\times [1/2,1] \to (\R^n)^k$ by $G=H\circ \xi$. Since $\xi$ is both Lipschitz and quasiconformal, we have that $G$ is a Lipschitz map having the property that $G|_{\R^{n-1}\times (1/2,1)}$ is a $K$-quasiregular $\vol_{(\R^n)^k}^\times$-curve, for $K=K(n)\ge 1$,
\[
g_i(x,1) = h_i(x,1) = A(x),
\]
and
\[
g_i(x,1/2) = h_i(x,1/2^k) = \frac{1}{2^k}A(2x)
\]
for each $x\in \R^{n-1}$ and $i\in \{1,\ldots, k\}$. This concludes the proof of the lemma.
\end{proof}


\section{Failure of discreteness of quasiregular $\omega_\sym$-curves of small distortion}
\label{sec:Rosay}

In this section, we show that Rosay's map gives an example of quasiregular $\omega_\sym$-curves of small distortion that are not discrete.

\begin{theorem}[Rosay {\cite[Proposition 1.2]{Rosay-ANF-2010}}]
\label{thm:Rosay}
For each $K>1$ and $k\ge 2$, there exists a non-constant $K$-quasiregular $\omega_\sym$-curve $\C\to \C^k$, which is not discrete.
\end{theorem}

For the reader's convenience, we begin by sketching Rosay's original construction of a non-discrete map $u=(u_1,u_2)\colon B^2\to \C^2$ satisfying $|\partial_{\bar z} u(z)|\le \varepsilon(z) |\partial u(z)|$ with $\varepsilon(z)\to 0$ as $z\to 0$.

The coordinate functions $u_i \colon B^2 \to \C$ of $u$ are defined as follows. Let $n\in \N_+$. Let 
\[
2^{-n} < r_n = \frac{5}{4} 2^{-n} < a_n = \frac{3}{2} 2^{-n} < R_n = \frac{7}{4} 2^{-n} < 2^{-n+1}
\]
and let $A_n$ be the closed annulus
\[
A_n = \{ z\in \C \colon 2^{-n} \le |z| \le 2^{-n+1} \}.
\]

Let $\psi_n \in C^\infty(\C)$ be a function satisfying $0\le \psi_n \le 1$, $\psi_n \equiv 1$ in $B^2(2^{-n})$, $\psi_n \equiv 0$ in $\C \setminus B^2(r_n)$, and $|d\psi_n|\le 6\cdot 2^n$. Similarly, let $\varphi_n \in C^\infty(\C)$ be a function satisfying $0\le \varphi_n \le 1$, $\varphi_n \equiv 1$ in $B^2(R_n)$, $\varphi_n \equiv 0$ in $\C \setminus B^2(2^{-n+1})$, and $|d\varphi_n|\le 6\cdot 2^n$. Let also $\chi_n$ be the function $1-\varphi_n$; see Figure \ref{fig:Rosay}.

\begin{figure}[h!]
\begin{tikzpicture}[thick, scale=0.7]
\draw[->] (-1,0) -- (13,0);
\draw[->] (-1,0) -- (-1,4);
\draw (-1 cm,2pt) -- (-1 cm,-2pt) node[anchor=north] {$0$};
\draw (1 cm,2pt) -- (1 cm,-2pt) node[anchor=north] {$2^{-n}$};
\draw (11 cm,2pt) -- (11 cm,-2pt) node[anchor=north] {$2^{-n+1}$};
\draw (3.5 cm,2pt) -- (3.5 cm,-2pt) node[anchor=north] {$r_n$};
\draw (6 cm,2pt) -- (6 cm,-2pt) node[anchor=north] {$a_n$};
\draw (8.5 cm,2pt) -- (8.5 cm,-2pt) node[anchor=north] {$R_n$};
\draw (-0.9 cm,3 cm) -- (-1.1 cm,3 cm) node[anchor=east] {$1$};
\draw plot [smooth] coordinates {(-1,3) (1.1,2.95) (1.7,2.5) (2.5,0.5) (3.4,0)};
\node at (2,3.2) {$\psi_n$};
\draw plot [smooth] coordinates {(13,3) (10.9,2.95) (10.3,2.5) (9.5,0.5) (8.6,0)};
\node at (10,3.2) {$\chi_n$};
\end{tikzpicture}
\caption{Functions $\psi_n$ and $\chi_n$.}
\label{fig:Rosay}
\end{figure}
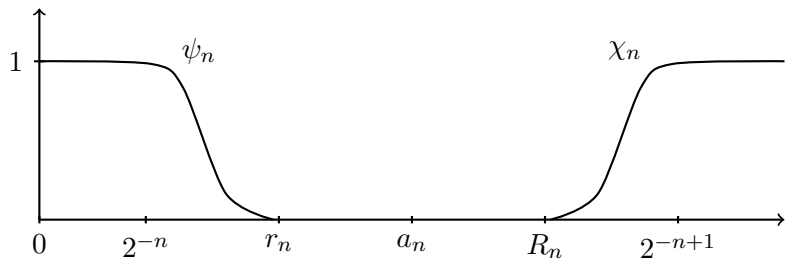

For $n$ even, we set
\[
u_1(z) = 2^\frac{n^2}{2} z^{n-1} (z-a_n)
\]
and
\[
u_2(z) = \chi_n(z) 2^\frac{(n-1)^2}{2} z^{n-2} (z-a_{n-1}) + \psi_n(z) 2^\frac{(n+1)^2}{2} z^n (z-a_{n+1})
\]
for every $z\in A_n$.

For $n$ odd, we set
\[
u_1(z) = \chi_n(z) 2^\frac{(n-1)^2}{2} z^{n-2} (z-a_{n-1}) + \psi_n(z) 2^\frac{(n+1)^2}{2} z^n (z-a_{n+1})
\]
and
\[
u_2(z) = 2^\frac{n^2}{2} z^{n-1} (z-a_n)
\]
for every $z\in A_n$.

Now both functions $u_i$ are defined in $B^2 \setminus \{0\}$. To obtain functions defined in $B^2$, we set $u_1(0)=u_2(0)=0$.

Now $u=(u_1,u_2) \colon B^2 \to \C^2$ is a $C^\infty$-smooth map. For each $n$, we have that $u_1(a_n)=0$ and $u_2(a_n)=0$. Thus, the map $u$ has a non-isolated zero at the origin. Additionally, there exists a constant $C>0$ for which
\begin{equation}
\label{eq:Rosay-C}
|\partial_{\bar z} u(z)| \le \frac{C}{n} |\partial_z u(z)|
\end{equation}
for every $z\in A_n$ when $n$ is large enough. We refer to Rosay \cite[Section 4]{Rosay-ANF-2010} for the estimates.

Having Rosay's construction at our disposal, it suffices to prove the following lemma, which interprets the complex dilatation of a map $\Omega \to \C^k$ in terms of the distortion as a quasiregular $\omega_\sym$-curve.

\begin{lemma}
\label{lemma:complex-dilatation}
Let $K\ge 1$, $\Omega\subset \C$ a domain, and $F=(f_1,\ldots,f_k) \colon \Omega \to \C^k$. If
\begin{equation}
\label{eq:complex-dilatation}
|\partial_{\bar z} F| \le \frac{K-1}{K+1} |\partial_z F| \text{ a.e.~in } \Omega,
\end{equation}
then
\begin{equation}
\label{eq:omega_0-distortion}
\norm{DF}^2 \le K(\star F^* \omega_\sym) \text{ a.e.~in } \Omega.
\end{equation}
\end{lemma}

\begin{remark}
It is well-known that, if $k=1$, a mapping satisfies the quasiconformality conditions \eqref{eq:complex-dilatation} and \eqref{eq:omega_0-distortion} with the same constant $K$. This is not true for $k>1$. For example, consider the map $F\colon \C \to \C^2$, $z\mapsto (z,z+\bar z)$, which is $5$-quasiregular $\omega_\sym$-curve and satisfies \eqref{eq:complex-dilatation} with constant $K=(\sqrt{2}+1)/(\sqrt{2}-1)$.
Note, however, that if $F$ satisfies \eqref{eq:omega_0-distortion} with constant $K$, then $|\partial_{\bar z} F| \le \sqrt{(2K-1)/(2K+1)} |\partial_z F|$ almost everywhere.
\end{remark}

\begin{proof}[Proof of Lemma \ref{lemma:complex-dilatation}]
Let $z\in \Omega$ be a point for which \eqref{eq:complex-dilatation} holds. We may assume that $|Df(z)|\ne 0$. Then
\[
(K+1) |\partial_{\bar z} F(z)|\le (K-1)|\partial_z F(z)|
\]
and hence 
\[
|\partial_z F(z)| + |\partial_{\bar z} F(z)| \le 
K(|\partial_z F(z)| - |\partial_{\bar z} F(z)|) = 
K\frac{|\partial_z F(z)|^2 - |\partial_{\bar z} F(z)|^2}{|\partial_z F(z)| + |\partial_{\bar z} F(z)|}.
\]
Thus we have that
\begin{align*}
\norm{DF(z)}^2 &\le (|\partial_z F(z)| + |\partial_{\bar z} F(z)|)^2 \le K(|\partial_z F(z)|^2 - |\partial_{\bar z} F(z)|^2) \\
&= K \sum_{i=1}^k |\partial_z f_i(z)|^2 - |\partial_{\bar z} f_i(z)|^2 
= K (\star F^* \omega_\sym).
\end{align*}
The claim follows.
\end{proof}

\subsection*{Proof of Theorem \ref{thm:Rosay}}

Let $K>1$. Let $C$ be the constant in \eqref{eq:Rosay-C} and let $n_0$ be such that
\[
\frac{C}{n_0} \le \frac{K-1}{K+1}.
\]
We may assume that Rosay's inequality holds for $n\ge n_0$ for $n_0$ even.

Now let $F\colon \C \to \C^k$ be the map $F=(f_1,f_2,0,\ldots,0)$, where $f_1 \colon \C \to \C$ is the function
\[
z\mapsto \begin{cases}
u_1(z), & \text{ if } |z|\le a_{n_0}, \\
2^\frac{n_0^2}{2} z^{n_0-1} (z-a_{n_0}), & \text{ otherwise},
\end{cases}
\]
and $f_2 \colon \C \to \C$ is the function
\[
z\mapsto \begin{cases}
u_2(z), & \text{ if } |z|\le a_{n_0}, \\
0, & \text{ otherwise}.
\end{cases}
\]
Note that $F|_{B^2(a_{n_0})} = (u|_{B^2(a_{n_0})},0,\ldots, 0) \colon B^2(a_{n_0}) \to \C^k$, where $u\colon B^2\to \C^2$ is Rosay's map.

By Rosay's construction, the map $F$ is smooth and the preimage $F^{-1}(0)$ is not discrete. It remains to show that $F$ is a $K$-quasiregular $\omega_\sym$-curve.

Let $z\in \C$. If $|z|>a_{n_0}$, then $F$ is analytic at $z$ and hence $\partial_{\bar z} F(z)=0$. Now suppose that $|z|\le a_{n_0}$. Then $z\in A_n$ for some $n\ge n_0$ and $F=(u,0,\ldots,0)$ near $z$, where $u$ is Rosay's map. Thus
\[
|\partial_{\bar z} F| \le \frac{C}{n} |\partial_z F| \le \frac{C}{n_0} |\partial_z F| \le \frac{K-1}{K+1} |\partial_z F|
\]
in $\C$ by \eqref{eq:Rosay-C}. 
The claim follows by Lemma \ref{lemma:complex-dilatation}.


\section{Preliminaries on calibrations $\vol_{(\R^n)^k}$}
\label{sec:linear-algebra}

In this section, we recall elementary facts on the comass of $\vol_{(\R^n)^k}^\times$ and linear maps $L\colon \R^n \to (\R^n)^k$ satisfying
\[
\norm{L}^n = \star L^*\vol_{(\R^n)^k}^\times.
\]

Recall that the comass norm of $\omega\in \wedge^n V^*$ is given by
\[
\norm{\omega} = \max\{ \omega(v_1,\ldots, v_n) \colon |v_i|\le 1\} = \max_L \star L^*\omega,
\]
where the second maximum is taken over all linear maps $L\colon \R^n \to V$ satisfying $|L(e_i)|\le 1$ for $i\in \{1,\ldots, n\}$. Here, the optimal linear maps are conformal. We rephrase this fact as follows.

\begin{lemma}
\label{lemma-linearmaprigidity}
Let $\omega \in \wedge^n V$ be a calibration. Then each linear map $L \colon \R^n \to V$ satisfying
\[
\norm{L}^n = \star L^*\omega
\]
is conformal.
\end{lemma}

\begin{proof}
By the singular value decomposition, we have that $L = Q D P$, where $Q\colon \R^n \to V$ and $P\colon \R^n \to \R^n$ are isometries, and $D \colon \R^n \to \R^n$ is a diagonal map whose diagonal elements are singular values of $L$. Then 
\[
L^*\omega = (QDP)^*\omega = P^*D^*Q^*\omega = (\det P)(\det D) Q^*\omega
\]
and 
\[
\norm{L} = \norm{QDP} = \norm{D}.
\]
Thus 
\[
\norm{D}^n = \norm{L}^n = \star L^*\omega = (\det P)(\det D)\star Q^*\omega \le (\det D) \norm{Q}^n = \det D.
\]
We conclude that $D = \norm{D} \id$ and hence $L$ is conformal.
\end{proof}

The classification of the maximal planes of $\vol_{(\R^n)^k}^\times$ is well-known. For the reader's convenience we give, in the following lemma, a simple proof which simultaneously shows that the forms $\vol_{(\R^n)^k}^\times$ are calibrations. 

\begin{lemma}
\label{lemma:linear-rigidity}
For each $n\ge 2$ and $k\ge 1$, the form $\vol_{(\R^n)^k}^\times$ is a calibration. Moreover, if a linear map $L=(L_1,\ldots, L_k)\colon \R^n \to (\R^n)^k$ satisfies
\[
\norm{L}^n = \star L^*\vol_{(\R^n)^k},
\]
then each linear map $L_i$ is conformal and, if $n\ge 3$, there exists $i_0\in \{1,\ldots, n\}$ for which $L_i=0$ for $i\ne i_0$.
\end{lemma}

\begin{proof}
As a preliminary step, for each $n\ge 2$, let $\psi_n \colon [0,1]^n \to \R$ be the function $(x_1,\ldots, x_n) \mapsto \sqrt{1-x_1^2} \cdots \sqrt{1-x_n^2} + x_1\cdots x_n$. For $n=2$, we have that $\max_x \psi_2(x) = 1$ and 
\[
E_2 = \{ x \in [0,1]^2 \colon \psi_2(x)=1\} = \{(t,t) \colon t\in [0,1]\}.
\]
For $n\ge 3$ and $(x_1,\ldots, x_n) \in (0,1)^n$, we have that
\[
\psi_n(x_1,\ldots, x_n) < \psi_{n-1}(x_1,\ldots, x_{n-1}) < \psi_2(x_1,x_2) \le \psi_2(x_1,x_1) = 1,
\]
where the first two inequalities follow trivially from the assumption that $x_3,\ldots,x_n \in (0,1)$. Thus $\max_x \psi_n(x) = 1$ for each $n\ge 3$ and we have that 
\[
E_n = \{ x \in [0,1]^n \colon \psi_n(x) = 1\} = \{ (0,\ldots, 0), (1,\ldots, 1)\}.
\]

We prove now both claims in the lemma simultaneously by induction in $k\in \N$. For $k=1$, we have that $\vol_{(\R^n)^k}^\times = \vol_{\R^n}$. Thus clearly $\norm{\vol_{(\R^n)^k}^\times} = 1$ and a linear mapping $L \colon \R^n \to \R^n$ satisfying $\norm{L}^n = \star L^*\vol_{\R^n}$ is conformal by Lemma \ref{lemma-linearmaprigidity}.

Suppose now that the claim holds for $k-1$. Let $L= (L_1,\ldots, L_k)\colon \R^n \to (\R^n)^k$ be a linear map satisfying $|L(e_i)|\le 1$ for each $i\in \{1,\ldots, n\}$. Let also $L' = (L_1,\ldots, L_{k-1}) \colon \R^n \to (\R^n)^{k-1}$. For each $j\in \{1,\ldots, n\}$, let also $x_j = |L_k(e_j)|\in [0,1]$. Then
\begin{equation}
\label{eq:L-vol-n-k}
\begin{split}
\star L^*\vol_{(\R^n)^k}^\times &= \star (L')^*\vol_{(\R^n)^{k-1}}^\times + \star L_k^* \vol_{\R^n} \\
&= \vol_{(\R^n)^{k-1}}^\times(L'(e_1),\ldots, L'(e_n)) + \vol_{\R^n}(L_k(e_1),\ldots. L_k(e_n)) \\
&\le \sqrt{1-x_1^2} \cdots \sqrt{1-x_n^2} \norm{\vol_{(\R^n)^{k-1}}^\times} + x_1\cdots x_n \norm{\vol_{\R^n}} \\
&\le \sqrt{1-x_1^2} \cdots \sqrt{1-x_n^2} + x_1\cdots x_n = \psi_n(x_1,\ldots, x_n) \le 1.
\end{split}
\end{equation}
Hence 
\[
\norm{\vol_{(\R^n)^k}^\times}= \max_L \star L^*\vol_{(\R^n)^k}^\times = 1.
\]

Let now $L=(L_1,\ldots, L_n) \colon \R^n \to (\R^n)^k$ be a linear map for which $\norm{L}^n = \star L^*\vol_{(\R^n)^k}^\times$. Then, by Lemma \ref{lemma-linearmaprigidity}, the map $L$ is conformal and we may further assume that $\norm{L(e_i)}=1$ for each $i\in \{1,\ldots, n\}$. Let $L'=(L_1,\ldots, L_{k-1}) \colon \R^n \to (\R^n)^{k-1}$ and let $x_j = |L_k(e_j)|$ for each $j\in \{1,\ldots, n\}$ as above. Since $\star L^*\vol_{(\R^n)^k} = 1$, we have that all inequalities in \eqref{eq:L-vol-n-k} are equalities. In particular, we have that $(x_1,\ldots, x_n)\in E_n$.

For $n\ge 3$, we have that $(x_1,\ldots, x_n)$ is either $(0,\ldots, 0)$ or $(1,\ldots, 1)$. In the first case, we have that $L_k =0$ and $\star (L')^*\vol_{(\R^n)^{k-1}}^\times = 1$. Then, by the induction assumption, there exists $i_0\in \{1,\ldots, k-1\}$ for which $L_i = 0$ for $i\ne i_0$. In the second case, we have that $L'=0$ and we may take $i_0=k$. 

For $n=2$, we have that $x_1 = x_2$. Since we have $\star (L')^*\vol_{(\R^n)^{k-1}}^\times \le \sqrt{1-x_1^2}\sqrt{1-x_2^2}$ and we have an equality in \eqref{eq:L-vol-n-k}, we obtain that
\[
\det L_k = \star L_k^*\vol_{\R^n} = x_1x_2 = |L_k(e_1)| |L_k(e_2)|.
\]
We have by the Hadamard inequality that $L_k$ is conformal. Since $L$ and $L_k$ are conformal, we conclude that 
\[
L'(e_1)\cdot L'(e_2) = L(e_1)\cdot L(e_2) - L_k(e_1)\cdot L_k(e_2) = 0.
\]
Since $|L'(e_1)| = \sqrt{1-x_1^2} = \sqrt{1-x_2^2} = |L'(e_2)|$, we have that $L'$ is conformal and $\norm{L'}^2 = \star (L')^*\vol_{(\R^2)^{k-1}}^\times$. Thus the linear maps $L_1,\ldots, L_{k-1}$ are conformal by the induction assumption.

This concludes the proof.
\end{proof}

\begin{remark}
For $n=1$ and $k\ge 1$, the comass norm depends on $k$ and we have that
\[
\norm{\vol_{(\R^1)^k}^\times} = \sqrt{k}.
\]
Indeed, for a linear map $L=(L_1,\ldots, L_k) \colon \R\to (\R^1)^k$, $t\mapsto (ta_1,\ldots, ta_k)$, satisfying $|L(1)| = |(a_1,\ldots, a_k)| \le 1$, we have that
\[
\star L^*\vol_{(\R^1)^k}^\times = \sum_{i=1}^k \vol_{\R}(L_i(1)) = \sum_{i=1}^k a_i \le \sqrt{k} \sqrt{a_1^2+ \cdots + a_k^2} = \sqrt{k}.
\]
Since the upper bound $\sqrt{k}$ is reached, we conclude that $\norm{\vol_{(\R^1)^k}^\times}=\sqrt{k}$.
\end{remark}


\section{Calibrated curves are $n$-energy minimizing}
\label{sec:n-harmonic}

In this section we show that calibrated curves with Euclidean targets are $n$-harmonic and hence $C^{1,\alpha}$-H\"older continuous. The $n$-harmonicity follows along the same lines as quasiminimality of quasiregular curves; see \cite[Section 2]{Pankka-AASF-2020}. The main difference is that in \cite{Pankka-AASF-2020} the energy of the map is understood in terms of the operator norm instead of the Hilbert--Schmidt norm. Recall that the Hilbert--Schmidt norm of a linear map $L \colon V\to W$ is 
\[
\norm{L}_\HS = \left( \sum_{i=1}^n |L(e_i)|^2 \right)^{1/2}
\]
where $(e_1,\ldots, e_n)$ is an orthonormal basis of $V$. Note that
\[
\norm{L} \le \frac{1}{n^{1/2}} \norm{L}_{\HS},
\]
where the inequality is an equality if $L$ is conformal. 

We say that a mapping $u \colon M\to N$ between Riemannian manifolds in $W_{\loc}^{1,n}(M,N)$ is a \emph{local $n$-minimizer} if for each compact submanifold $G\subset M$ with boundary and each $W^{1,n}(G,N)$-mapping $v\colon G \to N$ having the same trace at $\partial G$ as $u$, we have that 
\begin{equation}
\label{eq:n-minimizer}
\int_{G} \norm{Du}_\HS^n \le \int_{G} \norm{Dv}_\HS^n;
\end{equation}
see Hardt and Lin \cite{Hardt-Lin-CPAM-1987} for the terminology. 

It is easy to show that calibrated curves $M\to \R^m$ are local $n$-minimizers and hence $C^{1,\alpha}$-regular by a result of Hardt and Lin \cite[Corollary 3.2]{Hardt-Lin-CPAM-1987}.

\begin{theorem}
\label{thm:Hardt-Lin}
Let $M$ be a Riemannian $n$-manifold, $\omega \in \Omega^n(\R^m)$ a calibration, and $F\colon M\to \R^m$ an $\omega$-calibrated curve. Then $F$ is a local $n$-minimizer. In particular, $F$ is $C^{1,\alpha}$-regular for $\alpha \in (0,1)$.
\end{theorem}

\begin{proof}
Let $G \subset M$ be a compact submanifold with boundary. Suppose first that $v \colon G\to \R^m$ is a continuous map in $W^{1,n}(G,\R^m)$ satisfying $v|_{\partial G} = F|_{\partial G}$. Since $\R^m$ is contractible, the map $v$ is homotopic to $F|_G \colon G\to \R^m$ (rel $\partial G$). Thus 
\[
\int_G F^*\omega = \int_G v^*\omega.
\]
Since $F$ is $\omega$-calibrated, we have that $\frac{1}{n^{n/2}} \norm{DF}^n_\HS = \star F^*\omega$ almost everywhere in $M$. Thus
\[
\frac{1}{n^{n/2}} \int_G \norm{DF}^n_\HS = \int_G F^*\omega = \int_G v^*\omega \le \int_G \norm{Dv}^n \le \frac{1}{n^{n/2}} \int_G \norm{Dv}^n_\HS.
\]
The general case of a test mapping $v\in W^{1,n}(G,N)$, having the same trace at $\partial G$ as $F$, follows now by a standard convolution argument.
\end{proof}

The Euclidean version of Theorem \ref{thm:manifold-Liouville} is now an immediate consequence of  Theorem \ref{thm:Hardt-Lin}.
\begin{theorem}
\label{thm:Euclidean-Liouville}
Let $n\ge 3$ and $k\ge 1$. Let $F=(f_1,\ldots, f_n) \colon \Omega\to (\R^n)^k$ be a $\vol_{(\R^n)^k}^\times$-calibrated curve, where $\Omega\subset \R^n$ is a domain. Then there exists an index $i_0\in \{1,\ldots, k\}$ for which the coordinate mapping $f_{i_0} \colon \Omega \to \R^n$ is a restriction of a M\"obius transformation $\bS^n \to \bS^n$ and each $f_i \colon \Omega \to \R^n$ for $i\ne i_0$ is a constant map.
\end{theorem}

\begin{proof}
We may assume that $F$ is non-constant. By Theorem \ref{thm:Hardt-Lin}, the differential $DF$ is continuous. Let 
\[
\widetilde \Omega = \{ x\in \Omega \colon DF(x)\ne 0\}.
\]
Since $F$ is non-constant, we have that $\widetilde \Omega$ is a non-empty open set.

For each $x\in \widetilde \Omega$, there exists, by continuity of $DF$ and Lemma \ref{lemma:linear-rigidity}, a connected neighborhood $U_x \subset \widetilde \Omega$ of $x$ and an index $i_x\in \{1,\ldots, k\}$ for which $Df_{i_x}$ is non-vanishing in $U_x$ and $Df_i|_{U_x} =0$ for $i\ne i_0$. Since $f_{i_x}$ is conformal in $U_x$, we have that it is M\"obius by the classical Liouville's theorem.

Let now $G \subset \widetilde \Omega$ be a component of $\widetilde \Omega$. Then $\{ U_x \}_{x\in G}$ is a covering of $G$. Since $G$ is connected, we conclude that there exist a unique index $i_G\in \{1,\ldots, k\}$ for which $f_{i_G}|_G$ is a non-constant M\"obius transformation and that each $f_i$ is constant in $G$ for $i\ne i_G$. Let $f\colon \bS^n\to \bS^n$ be the unique M\"obius transformation extending $f_{i_G}|_G$. 

Since the zero set of $DF$ is contained in the zero set of $Df$, which is empty, we conclude that $\widetilde \Omega = \Omega$. Thus $G=\Omega$. The claim is proven.
\end{proof}

\begin{remark}
\label{rmk:example-holomorphic}
Similarly $\vol_{(\R^2)^k}^\times$-calibrated curves are holomorphic curves. Indeed, let $\Omega$ be a domain in $\R^2$ and let $F\colon (f_1,\ldots, f_k) \colon \Omega \to (\R^2)^k$ be a $\vol_{(\R^2)^k}^\times$-calibrated curve. Then, by Lemma \ref{lemma:linear-rigidity} we have that each $f_i$ has conformal differential. Hence each $f_i$ is a weakly conformal map in the Sobolev space $W^{1,2}_\loc(\Omega, \R^2)$. Since each $f_i$ is a weak solution to the Cauchy--Riemann equations, we have that $f_i$ is holomorphic by Weyl's lemma; see e.g.~Astala--Iwaniec--Martin \cite[Lemma A.6.10]{Astala-Iwaniec-Martin-book}.
\end{remark}


\section{Bounded sequences of quasiregular $\vol_{(\R^n)^k}^\times$-curves are normal}
\label{sec:normal-family}

In this section, we show that a bounded sequence of $K$-quasiregular $\vol_{(\R^n)^k}^\times$-curves is normal. Since locally uniform limits of $K$-quasiregular $\omega$-curves are $K$-quasiregular $\omega$-curves by \cite[Theorem 1.9]{Pankka-AASF-2020}, it suffices to show that a bounded sequence of $K$-quasiregular $\omega$-curves $\Omega \to \R^m$, where $\Omega \subset \R^n$ is a domain, has a converging subsequence when $\omega$ is a calibration with constant coefficients.

\begin{theorem}
\label{thm:normal}
Let $n\ge 2$ and let $\omega \in \wedge^n \R^m$ be a calibration. Let $\Omega \subset \R^n$ be a domain, let $K\ge 1$ and let $(F_k)_{k\in \N}$ be a bounded sequence of $K$-quasiregular $\omega$-curves $F_k\colon \Omega \to \R^m$. Then the sequence $(F_k)_{k\in \N}$ has a locally uniformly converging subsequence and the limiting map is a $K$-quasiregular $\omega$-curve.
\end{theorem}

The proof of Theorem \ref{thm:normal} is based on a quantitative version of local H\"older continuity for quasiregular curves. In \cite{Onninen-Pankka-arXiv-2020} it is shown, using the now standard Morrey's argument, that, for a calibration $\omega\in \wedge^n \R^m$, a $K$-quasiregular $\omega$-curve is locally $\alpha$-H\"older continuous with
\begin{equation}
\label{eq:Holder-alpha}
\alpha = \frac{1}{K} \frac{\norm{\omega}}{|\omega|_{\ell_1}} = \frac{1}{K} \frac{1}{|\omega|_{\ell_1}},
\end{equation}
where $|\omega|_{\ell_1}$ is the $\ell_1$-norm of $\omega$. It is not known whether the exponent $\alpha$ is sharp.

In what follows, we verify that the same argument yields also uniform estimates for the multiplicative H\"older constant for sequences of bounded curves. This immediately yields Theorem \ref{thm:normal}. For the initial step in the proof we quote the following Caccioppoli type estimate which states that, locally, the diameter of the image controls the energy of the curve; see \cite[Lemma 6.1]{Onninen-Pankka-arXiv-2020} for details.

\begin{lemma}
\label{lem:bounded-energy}
Let $n\ge 2$ and let $\omega \in \wedge^n \R^m$ be a calibration. Let $\Omega \subset \R^n$ be a domain, let $K\ge 1$ and let $F\colon \Omega \to \R^m$ be a $K$-quasiregular $\omega$-curve. 
Then
\[
\left( \int_{\frac{1}{2}B} \norm{DF}^n \right)^\frac{1}{n} \le C(n) K \diam(FB)
\]
for every ball $B\subset \R^n$ satisfying $\overline{B}\subset \Omega$.
\end{lemma}

\begin{proof}[Proof of Theorem \ref{thm:normal}]
By the Arzel\`a--Ascoli theorem, it suffices to show that the sequence $(F_k)$ is locally uniformly $\alpha$-H\"older. That is, there exists a constant $0<\alpha \le 1$ for which for every $x\in \Omega$ there exists a constant $C>0$ and a neighborhood $U$ satisfying
\[
|F_k(x)-F_k(y)| \le C|x-y|^\alpha
\]
for every $y\in U$ and every $k\in \N$.

Let $x\in \Omega$ and let $R>0$ be such that $\overline{B^n(x,8R)} \subset \Omega$. Let also $a\in B^n(x,R)$ and $r>0$ be such that $B_r = B^n(a,r) \subset B^n(x,2R)$.

Unraveling the statement of \cite[Lemma 5.1]{Onninen-Pankka-arXiv-2020}, we obtain that, for $\alpha\in (0,1)$ in \eqref{eq:Holder-alpha},
we have that
\begin{align*}
\left( \frac{1}{|B_r|} \int_{B_r} \norm{DF_k}^n  \right)^\frac{1}{n} 
&\le \frac{1}{|B(0,1)|^\frac{1}{n}} \frac{r^{\alpha-1}}{R^{\alpha}}  \left( \int_{B^n(x,3R)} \norm{DF_k}^n \right)^\frac{1}{n}
\end{align*}
for every $k\in \N$.

Let now
\[
L = \sup_k \diam(F_k\Omega)<\infty.
\]
Then, by Lemma \ref{lem:bounded-energy}, we have that 
\[
\left( \frac{1}{|B_r|} \int_{B_r} \norm{DF_k}^n  \right)^\frac{1}{n} \le C(n,R,\alpha) KL r^{\alpha-1}
\]
for every $k\in \N$.

We may now apply standard chain argument (see Haj\l asz--Koskela \cite{Hajlasz-Koskela-SobolevmetPoincare} for the general result or \cite[Lemma 5.2]{Onninen-Pankka-arXiv-2020} for the special case), to obtain
\begin{align*}
|F_k(x)-F_k(y)| \le C(n,R,\alpha,K,L) |x-y|^\alpha \left( \int_{B^n(x,4R)} \norm{DF_k}^n \right)^\frac{1}{n} 
\end{align*}
for every $y\in B^n(x,R)$ and every $k\in \N$. By applying Lemma \ref{lem:bounded-energy} again, we have that
\[
|F_k(x)-F_k(y)| \le C(n,R,K,L,\omega)|x-y|^\alpha
\]
for every $y\in B^n(x,R)$ and every $k\in \N$.
\end{proof}


\section{Local injectivity of $\vol_{(\R^n)^k}^\times$-curves of small distortion}
\label{sec:Iwaniec}

In this section, we prove that quasiregular $\vol_{(\R^n)^k}^\times$-curves of small distortion are local quasiconformal embeddings.

\begin{theorem}
\label{thm:local-injectivity}
Let $n\ge 3$, $k\ge 1$ and $H>1$. Then there exists $\varepsilon=\varepsilon(n,k,H)>0$ for the following. Let $\Omega \subset \R^n$ be a domain and let $F\colon \Omega\to (\R^n)^k$ be a non-constant $(1+\varepsilon)$-quasiregular $\vol_{(\R^n)^k}^\times$-curve. Then $F$ is a local $H$-quasiconformal embedding.
\end{theorem}

Our proof follows the proof of Iwaniec \cite{Iwaniec-PAMS-1987} for the Martio--Rickman--V\"ais\"al\"a theorem on the local injectivity of the quasiregular mappings of small distortion; see Martio--Rickman--V\"ais\"al\"a \cite[Theorem 4.6]{MRV-AASF-1971} or Rickman \cite[Theorem VI.8.14]{Rickman-book}.  See also Kopylov \cite{Kopylov-Siberian-1982} for a functional approach to the stability of Liouville's theorem.

We divide the discussion into three parts. For the proof, we introduce first sheaves of quasiregular curves and consider an associated gauge function measuring the distance of quasiregular $\vol_{(\R^n)^k}^\times$-curves to $\vol_{(\R^n)^k}^\times$-calibrated curves. In the second step, we develop a Harnack estimate for quasiregular  $\vol_{(\R^n)^k}^\times$-curves of small distortion. Finally, in the third step, we combine the distance estimate in terms of the gauge function and the Harnack estimate to obtain the local injectivity and metric quasiconformality of curves of small distortion.

\subsection{Sheaves $\ccF_t^{n,k}$ of quasiregular $\vol_{(\R^n)^k}^\times$-curves}

We define the following auxiliary sheaves $\ccF_t^{n,k}$ of quasiregular curves. For a general discussion of sheaves on manifolds; see \cite[Chapter 1.6]{forster-riemann}.  Let $n\ge 3$ and $k\ge 1$. For each $t\ge 0$, let $\ccF_t^{n,k}$ be the sheaf of $(1+t)$-quasiregular $\vol_{(\R^n)^k}^\times$-curves on $\R^n$, that is, for each open set $\Omega \subset \R^n$, let $\ccF_t^{n,k}(\Omega)$ be the family of all $(1+t)$-quasiregular $\vol_{(\R^n)^k}^\times$-curves $\Omega \to (\R^n)^k$. Note that $\ccF_t^{n,k}$ is a sheaf, since a restriction of a $K$-quasiregular $\vol_{(\R^n)^k}^\times$-curve into a subdomain is a $K$-quasiregular $\vol_{(\R^n)^k}^\times$-curve and two $K$-quasiregular $\vol_{(\R^n)^k}^\times$-curves $\Omega \to (\R^n)^k$ and $\Omega'\to (\R^n)^k$, which agree on $\Omega\cap \Omega'$, define a $K$-quasiregular $\vol_{(\R^n)^k}^\times$-curve $\Omega\cup \Omega'\to (\R^n)^k$.

Let also $\ccF^{n,k}$ be the sheaf of $\vol_{(\R^n)^k}^\times$-calibrated curves. Then $\ccF_0^{n,k}=\ccF^{n,k}$. Note that, by Theorem \ref{thm:Euclidean-Liouville}, we have that $\ccF^{n,k}$ consists of constant mappings $\Omega\to (\R^n)^k$ and mappings $\varphi=(\varphi_1,\ldots,\varphi_k) \colon \Omega \to (\R^n)^k$ for which there exists $i_0\in \{1,\ldots,k\}$ so that $\varphi_{i_0} \colon \Omega \to \R^n$ is a restriction of a M\"obius transformation $\bS^n \to \bS^n$ and each $\varphi_i \colon \Omega \to \R^n$ is a constant map for $i\ne i_0$. Note that, in particular, $\ccF^{n,1}$ is the family of constant mappings $\Omega \to \R^n$ and restrictions of M\"obius transformations $\bS^n \to \bS^n$.

The sheaves $\ccF_t^{n,k}$ are also closed under translation and dilation and, by Theorem \ref{thm:normal}, they are also closed under locally uniform limits. We record these properties as lemmas. The proof follows from basic properties of quasiregular curves and we omit the details.

\begin{lemma}
Let $n\ge 3$ and $k\ge 1$. Let $F\in \ccF_t^{n,k}(\Omega)$, where $t\ge 0$ and $\Omega \subset \R^n$ is a domain. Let $x_0 \in \R^n$, $y_0 \in (\R^n)^k$, $\lambda >0$ and $\mu >0$. Then the map $x\mapsto \lambda F(\mu x + x_0) + y_0$ belongs to $\ccF_t^{n,k}(\widetilde{\Omega})$, where $\widetilde{\Omega} = \{ \frac{1}{\mu}(x-x_0) \colon x\in \Omega\}$.
\end{lemma}

\begin{lemma}
\label{lemma:**}
Let $n\ge 3$, $k\ge 1$, and let $\Omega\subset \R^n$ be a domain. Then, for a sequence $t_i \searrow 0$ and a bounded sequence $(F_i)_{i\in \N}$ of mappings $F_i \in \ccF_{t_i}^{n,k}(\Omega)$,
there exists a subsequence $(F_{i_j})_{j\in \N}$ and a mapping $F\in \ccF^{n,k}(\Omega)$ for which $F_{i_j}\to F$ locally uniformly in $\Omega$.
\end{lemma}

Having sheaves $\ccF_t^{n,k}$ at our disposal, we may now define, as in \cite{Iwaniec-PAMS-1987}, an auxiliary gauge function $\tau^{n,k} \colon [0,\infty) \to [0,\infty)$ by setting
\[
\tau^{n,k}(t) = \sup_{\substack{0\le s\le t \\ F\in \ccF_s^{n,k}(B^n) \\ \sup_B |F| \le 1}} \min_{\substack{\varphi \in \ccF^{n,k}(B^n) \\ \varphi(0)=F(0)}} d(F,\varphi)
\]
for $t\ge 0$, where
\[
d(F,\varphi)=\sup_{x\in B^n} \, (1-|x|) |F(x)-\varphi(x)|.
\]

Note that the minimum in the definition of $\tau^{n,k}$ is well-defined by Lemma \ref{lemma:**}. Also $\lim_{t\to 0} \tau^{n,k}(t)=0$; see \cite{Iwaniec-PAMS-1987} for details.

The gauge function $\tau^{n,k}$ serves several purposes in the proof of Theorem \ref{thm:local-injectivity}. An immediate application is to approximate curves in $\ccF_t^{n,k}$ by calibrated curves. Since the proof of the following proposition is verbatim to the proof of \cite[Lemma 1]{Iwaniec-PAMS-1987}, we merely recall the statement here. 

\begin{proposition}[{\cite[Lemma 1]{Iwaniec-PAMS-1987}}]
\label{prop:1-qr-near}
Let $n\ge 3$ and $k\ge 1$. Let $B=B^n(a,r)\subset \R^n$ be a ball. Let also $t\ge 0$ and $F\in \ccF_t^{n,k}(B)$. Then there exists $\varphi \in \ccF^{n,k}(B)$ for which $\varphi(a)=F(a)$ and
\[
\sup_{\sigma B} |F-\varphi| \le \frac{\tau^{n,k}(t)}{1-\sigma} \sup_B |F|
\]
for every $\sigma \in [0,1)$.
\end{proposition}

\subsection{Harnack estimate}

The following Harnack estimate for curves of small distortion, in terms the gauge function $\tau^{n,k}$, is the basis of the forthcoming distortion estimates.

\begin{proposition}
\label{prop-iwaniecstabilitylemma3}
Let $n\ge 3$ and $k\ge 1$. Let $\Omega \subset \R^n$ be a domain. Let also $t\ge 0$ and $F\in \ccF_t^{n,k}(\Omega)$. Suppose that $\tau^{n,k}(t)\le 10^{-3}$. Then
\[
\sup_{\frac{1}{2} B} |F| \le 2^\frac{4-4\rho}{\rho} \sup_{\rho B} |F|
\]
for every ball $B\subset \Omega$ and every $0<\rho \le \frac{1}{2}$.
\end{proposition}

Since the proof of Proposition \ref{prop-iwaniecstabilitylemma3} is based on the analysis of the norm of the mapping $F$, that is, the function $x\mapsto |F(x)|$, the proof is essentially the same as the proof of \cite[Lemma 3]{Iwaniec-PAMS-1987}. For this reason, we merely discuss briefly the proof of \cite[Proposition 2]{Iwaniec-PAMS-1987}, where the particular structure of $\vol_{(\R^n)^k}^\times$-calibrated curves have a role, and omit other details; see also \cite[Lemmas 2 and 3]{Iwaniec-PAMS-1987}. The following proposition is a restatement of \cite[Proposition 2]{Iwaniec-PAMS-1987}. 

\begin{proposition}
\label{prop:triangle-ineq}
Let $n\ge 3$, $k\ge 1$, and let $\Omega \subset \R^n$ be a domain. Let also $t\ge 0$ and $F\in \ccF_t^{n,k}(\Omega)$. Then
\[
\sup_{\frac{4}{5} B} |F| \le 8\sup_{\frac{2}{5} B} |F| + 19\tau^{n,k}(t)\sup_B |F|
\]
for every ball $B\subset \Omega$.
\end{proposition}

\begin{proof}
Let $B\subset \Omega$ be a ball. By Proposition \ref{prop:1-qr-near}, there exists $\varphi \in \ccF^{n,k}(B)$ for which
\[
\sup_{\sigma B} |F-\varphi|\le \frac{\tau^{n,k}(t)}{1-\sigma} \sup_B |F|
\]
for every $\sigma \in [0,1)$. Then, by the triangle inequality, we have the estimate
\[
\sup_{\frac{4}{5} B} |F|\le \sup_{\frac{4}{5} B} |F-\varphi| + \sup_{\frac{4}{5} B} |\varphi| \le 5\tau^{n,k}(t)\sup_B |F| + \sup_{\frac{4}{5} B} |\varphi|.
\]
Since $\varphi=(\varphi_1,\ldots,\varphi_k) \colon \Omega \to (\R^n)^k$, where either $\varphi$ is constant or exactly one coordinate map $\varphi_i$ is a restriction of a M\"obius transformation $\bS^n \to \bS^n$ and the other coordinate maps $\varphi_j$ are constant, we have that
\[
\sup_{\frac{4}{5} B} |\varphi| \le 8\sup_{\frac{2}{5} B} |\varphi|;
\]
see the proof of \cite[Proposition 2]{Iwaniec-PAMS-1987} for details. Then
\[
8\sup_{\frac{2}{5} B} |\varphi| \le 8\sup_{\frac{2}{5} B} |F-\varphi| + 8\sup_{\frac{2}{5} B} |F| \le 14\tau^{n,k}(t)\sup_B |F| + 8\sup_{\frac{2}{5} B} |F|
\]
and the claim follows.
\end{proof}

\subsection{Local injectivity and metric quasiconformality}

Theorem \ref{thm:local-injectivity} follows almost immediately from the following local distortion estimate.

\begin{proposition}
\label{prop:iwaniec-distortion}
Let $n\ge 3$, $k\ge 1$, and $\delta >0$. There exists $\varepsilon=\varepsilon(n,k,\delta) >0$ for the following. Let $B=B^n(a,r)\subset \R^n$ be a ball and let $F\colon B\to (\R^n)^k$ be a $(1+\varepsilon)$-quasiregular $\vol_{(\R^n)^k}^\times$-curve. Then
\[
\sup_{x\in \rho B} |F(x)-F(a)| \le \left( \frac{1+2\rho}{1-2\rho} + 2^\frac{4}{\rho} \delta \right) \min_{|x-a|=\rho r} |F(x)-F(a)|
\]
for every $0<\rho \le \frac{1}{4}$.
\end{proposition}

\begin{proof}[Proof of Theorem \ref{thm:local-injectivity}] 

First, let $\rho=\rho(H) \in (0, \frac{1}{4})$ be such that
\[
\frac{1+2\rho}{1-2\rho} \le \frac{H+1}{2}.
\]
Second, fix $\delta=\delta(\rho,H) \in (0,2^{-16})$ for which 
\[
2^\frac{4}{\rho} \delta \le \frac{H-1}{2}.
\]

Let now $\varepsilon=\varepsilon(n,k,\delta) >0$ be as in Proposition \ref{prop:iwaniec-distortion}. Since $\delta=\delta(\rho,H)$ and $\rho=\rho(H)$, we have that $\varepsilon=\varepsilon(n,k,\delta)=\varepsilon(n,k,H)$.

Let $\Omega \subset \R^n$ be a domain and let $F\colon \Omega \to (\R^n)^k$ be a non-constant $(1+\varepsilon)$-quasiregular $\vol_{(\R^n)^k}^\times$-curve. To obtain that $F$ is a local embedding, it suffices to show that $F$ is locally injective.

Let $B=B^n(x,R)$ be a ball for which $B^n(x,9R) \subset \Omega$. Suppose that there exists distinct points $a,b\in B$ for which $F(a)=F(b)$. Then $B^n(a,4|a-b|)\subset \Omega$. Applying Proposition \ref{prop:iwaniec-distortion} with $\rho=1/4$ yields
\[
\max_{|x-a|\le |b-a|} |F(x)-F(a)| \le 4 \min_{|x-a|=|b-a|} |F(x)-F(a)| = 0.
\]
Thus, $F$ is constant in $B^n(a,|b-a|)$. By repeating the same argument, it follows that $F$ is constant in the whole domain $\Omega$. Hence, $F$ must be injective in $B$. It remains to show that $F$ is $H$-quasiconformal.

Let $a\in \Omega$. Let $r>0$ for which $B^n(a,9\frac{r}{\rho})\subset \Omega$. Then $F$ is injective in $B^n(a,\frac{r}{\rho})$ and $\min_{|x-a|=r} |F(x)-F(a)| >0$. Thus, we obtain the estimate
\[
\frac{\max_{|x-a|=r} |F(x)-F(a)|}{\min_{|x-a|=r} |F(x)-F(a)|} \le \frac{1+2\rho}{1-2\rho} + 2^\frac{4}{\rho} \delta \le H
\]
by applying Proposition \ref{prop:iwaniec-distortion} in the ball $B^n(a,\frac{r}{\rho})$. The claim follows.
\end{proof}

Since the following distortion estimate in terms of the the gauge function $\tau^{n,k}$ is used twice in the proof of Proposition \ref{prop:iwaniec-distortion}, we state it as a separate lemma.

\begin{lemma}
\label{lemma:iwaniec-lemma4}
Let $n\ge 3$, $k\ge 1$, and $t\ge 0$. Let also $B=B^n(a,r)$ be a ball in $\R^n$. Then, for each $F\in \ccF_t^{n,k}(B)$, we have that 
\begin{equation}
\label{eq:iwaniec-lemma4-1}
\left(1- 2^{\frac{4-\rho}{\rho}} \tau^{n,k}(t)\right) \sup_{x\in \rho B} |F(x)-F(a)| \le \frac{1+2\rho}{1-2\rho} \min_{|x-a|=\rho r} |F(x)-F(a)|
\end{equation}
for every $0<\rho \le \frac{1}{4}$. In particular, if $\tau^{n,k}(t)\le 2^{-17}$, then
\begin{equation}
\label{eq:iwaniec-lemma4-2}
\sup_{x\in \frac{1}{4}B} |F(x)-F(a)| \le 4 \min_{|x-a|=\frac{1}{4} r} |F(x)-F(a)|.
\end{equation}
\end{lemma}

\begin{proof}
Since \eqref{eq:iwaniec-lemma4-2} follows immediately from  \eqref{eq:iwaniec-lemma4-1}, it suffices to prove  \eqref{eq:iwaniec-lemma4-1}.

We may assume that $a=0$ and $F(a)=0$. By Proposition \ref{prop:1-qr-near}, since $F|_{\frac{1}{2} B} \in \ccF_t^{n,k}(\frac{1}{2} B)$, there exists $\varphi \in \ccF^{n,k}(\frac{1}{2} B)$ for which $\varphi(0)=0$ and
\[
\sup_{\frac{\sigma}{2} B} |F-\varphi|\le \frac{\tau^{n,k}(t)}{1-\sigma} \sup_{\frac{1}{2} B} |F|
\]
for every $\sigma \in [0,1)$. In particular,
\[
\sup_{\rho B} |F-\varphi|\le \frac{\tau^{n,k}(t)}{1-2\rho} \sup_{\frac{1}{2} B} |F| \le 2\tau^{n,k}(t) \sup_{\frac{1}{2} B} |F|.
\]

Since $\varphi(0)=0$, we have either $\varphi \equiv 0$ or $\varphi=(0,\ldots,0,\varphi_{i_0},0\ldots,0)$, where $\varphi_{i_0}$ is a restriction of a M\"obius transformation $\bS^n \to \bS^n$. Thus $\varphi$ satisfies the estimate
\[
\sup_{\rho B} |\varphi| \le \frac{1+2\rho}{1-2\rho} \min_{|x|=\rho r} |\varphi(x)| \le \frac{1+2\rho}{1-2\rho} \left( \sup_{\rho B} |F-\varphi| + \min_{|x|=\rho r} |F(x)| \right).
\]
Then, by the triangle inequality and Proposition \ref{prop-iwaniecstabilitylemma3},
\begin{align*}
\sup_{\rho B} |F| &\le \sup_{\rho B} |F-\varphi| + \frac{1+2\rho}{1-2\rho} \left( \sup_{\rho B} |F-\varphi| + \min_{|x|=\rho r} |F(x)| \right) \\
&\le \frac{1+2\rho}{1-2\rho} \min_{|x|=\rho r} |F(x)| + 4\sup_{\rho B} |F-\varphi| \\
&\le \frac{1+2\rho}{1-2\rho} \min_{|x|=\rho r} |F(x)| + 8\tau^{n,k}(t) \sup_{\frac{1}{2} B} |F| \\
&\le \frac{1+2\rho}{1-2\rho} \min_{|x|=\rho r} |F(x)| + 2^\frac{4-\rho}{\rho} \tau^{n,k}(t) \sup_{\rho B} |F|. 
\end{align*}
The claim follows.
\end{proof}

We are now ready to prove Proposition \ref{prop:iwaniec-distortion}

\begin{proof}[Proof of Proposition \ref{prop:iwaniec-distortion}]
Since $\tau^{n,k}(t)\to 0$ as $t \to 0$, we may fix $\varepsilon >0$ for which $\tau^{n,k}(\varepsilon)\le \min \{ \delta/2, 2^{-17} \}$. Let now $B=B^n(a,r)$ be a ball in $\R^n$ and $F\colon B\to (\R^n)^k$ a $(1+\varepsilon)$-quasiregular $\vol_{(\R^n)^k}^\times$-curve. Let also $0<\rho \le \frac{1}{4}$.

Let $B'=B^n(a,4\rho r)$. Then $F|_{B'} \in \ccF_\varepsilon^{n,k}(B')$ and by \eqref{eq:iwaniec-lemma4-2}, we have that
\[
\sup_{x\in \rho B} |F(x)-F(a)| \le 4\min_{|x-a|=\rho r} |F(x)-F(a)|.
\]
Thus, by \eqref{eq:iwaniec-lemma4-1}, we have that
\begin{align*}
\sup_{x\in \rho B} |F(x)-F(a)| 
\le& \frac{1+2\rho}{1-2\rho} \min_{|x-a|=\rho r} |F(x)-F(a)| \\
&+ 2^\frac{4-\rho}{\rho} \tau^{n,k}(\varepsilon) \sup_{x\in \rho B} |F(x)-F(a)|\\
\le& \frac{1+2\rho}{1-2\rho} \min_{|x-a|=\rho r} |F(x)-F(a)| \\
&+ 2^\frac{4-\rho}{\rho} \cdot \frac{\delta}{2} \cdot 4 \min_{|x-a|=\rho r}|F(x)-F(a)|\\
\le& \left( \frac{1+2\rho}{1-2\rho} + 2^{\frac{4}{\rho}} \delta \right) \min_{|x-a|=\rho r} |F(x)-F(a)|.
\end{align*}
This completes the proof.
\end{proof}


\section{Almost maximal planes of $\vol_{(\R^n)^k}^\times$}
\label{sec:linear-algebra-2}

In the following proposition we characterize planes, in terms of linear maps, which almost maximize the comass norm of $\vol_{(\R^n)^k}^\times$. This proposition can be viewed as a linear algebraic version of Theorem \ref{thm:manifold-Liouville}.

\begin{proposition}
\label{prop:linear-1+epsilon}
Let $n\ge 3$,  $k \ge 1$, and $\varepsilon \in (0, 1/(100k))$. Then, for a linear map $L=(L_1,\ldots, L_k) \colon \R^n \to (\R^n)^k$ satisfying
\begin{equation}
\label{eq:linear-1+epsilon}
\norm{L}^n \le (1+\varepsilon) \left( \star L^*\vol_{(\R^n)^k}^\times \right),
\end{equation}
there exists an index $i_0\in \{1,\ldots k\}$ for which
\begin{equation}
\label{eq:qc-index}
\norm{L_{i_0}}^n \le (1+7k\sqrt{\varepsilon}) \det L_{i_0}.
\end{equation}
Moreover,
\begin{enumerate}
\item $\norm{L_{i_0}} \ge \norm{L}/(1+\varepsilon)$, and \label{item:good-index}
\item $\norm{L_i} \le 5\sqrt{k} \varepsilon^\frac{1}{4} \norm{L}$ for every $i\ne i_0$. \label{item:good-index-2}
\end{enumerate}
\end{proposition}

\begin{proof}
We may assume that $L$ is non-constant and that $\det L_i \ge 0$ for each $i\in \{1,\ldots, k\}$. Indeed, the claim holds trivially for constant maps and, by post-composing each $L_i$ with an orientation reversing isometry, if necessary, we obtain a linear map $\widetilde L = (\widetilde L_1, \ldots, \widetilde L_k) \colon \R^n \to (\R^n)^k$ satisfying \eqref{eq:linear-1+epsilon} with $\norm{\widetilde L} = \norm{L}$, $\star L^*\omega \le \star \widetilde L^*\omega$, and $\norm{\widetilde L_i}=\norm{L_i}$ for each $i\in \{1,\ldots, k\}$. We prove first \eqref{item:good-index}. Then we prove \eqref{eq:qc-index} and finally \eqref{item:good-index-2}.

For the proof of \eqref{item:good-index}, let $\lambda_i = \norm{L_i}$ for each $i=1,\ldots, k$. Since \eqref{eq:linear-1+epsilon} is invariant under permutation of the factors of $(\R^n)^k$, we may assume that $\lambda_1 \ge \lambda_2 \ge \cdots \ge \lambda_k \ge 0$. For \eqref{item:good-index} it suffices to show that $\lambda_1 \ge \norm{L}/(1+\varepsilon)$. 

For each $i\in \{1,\ldots, k\}$, let $Q_i\in \SO(n)$ be such an orientation preserving isometry that $Q_i(L_i(e_n)) = c_i e_n$, where $c_i\ge 0$, and denote $R_i = Q_iL_i \colon \R^n \to \R^n$. Let also $R = (R_1,\ldots, R_k) \colon \R^n \to (\R^n)^k$. Then
\begin{equation}
\label{eq:R-is-good}
\norm{R}^n \le (1+\varepsilon) \left( \star R^*\vol_{(\R^n)^k}^\times \right).
\end{equation}
Indeed, for each $v\in \R^n$, we have that
\[
|R(v)|^2 = \sum_{i=1}^k |R_i(v)|^2 = \sum_{i=1}^k |(Q_iL_i)(v)|^2 = \sum_{i=1}^k |L_i(v)|^2 = |L(v)|^2.
\]
Thus $\norm{R} = \norm{L}$. Since we also have that
\begin{align*}
\star R^*\vol_{(\R^n)^k}^\times 
&=\sum_{i=1}^k \star R_i^*\vol_{\R^n} 
= \sum_{i=1}^k \star L_i^*\vol_{\R^n} 
= \star L^*\vol_{(\R^n)^k}^\times,
\end{align*} 
we obtain \eqref{eq:R-is-good}.

For each $i=1,\ldots, k$, let $\widehat R_i = \pi\circ R_i|_{\R^{n-1}} \colon \R^{n-1}\to \R^{n-1}$, where $\pi \colon \R^n \to \R^{n-1}$ is the projection $(x_1,\ldots, x_n)\mapsto (x_1,\ldots, x_{n-1})$. Then, for each $i$, we have that $\det R_i = (\det \widehat R_i) c_i$. Let also $\widehat R = (\widehat R_1,\ldots, \widehat R_k) \colon \R^{n-1}\to (\R^{n-1})^k$.

Since $\lambda_1 \ge \lambda_i \ge c_i$ and $\det R_i\ge 0$ for each $i\in \{1,\ldots, k\}$, we have that
\begin{align*}
\norm{R}^n &\le (1+\varepsilon) \left( \star R^*\vol_{(\R^n)^k}^\times \right) = (1+\varepsilon) \sum_{i=1}^k \det R_i = (1+\varepsilon) \sum_{i=1}^k c_i \det \widehat R_i \\
&\le (1+\varepsilon) \lambda_1  \sum_{i=1}^k \det \widehat R_i
= (1+\varepsilon) \lambda_1 ( \star \widehat R^* \vol_{(\R^{n-1})^k}^\times) 
\le (1+\varepsilon) \lambda_1 \norm{\widehat R}^{n-1}. 
\end{align*}
Since $\norm{\widehat R}\le \norm{R}$, we conclude that
\begin{align*}
\lambda_1 \ge \frac{\norm{R}^n}{(1+\varepsilon) \norm{\widehat R}^{n-1}} \ge \frac{\norm{R}}{(1+\varepsilon)} = \frac{\norm{L}}{1+\varepsilon}.
\end{align*}
This concludes the proof of \eqref{item:good-index}.

To prove \eqref{eq:qc-index}, by precomposing $L$ with an isometry, we may assume that $\norm{L_1} = |L_1(e_1)|$. Let now $L'=(L_2,\ldots, L_k)\colon \R^n \to (\R^n)^{k-1}$. Then, by \eqref{item:good-index}, 
\begin{align*}
|L'(e_1)|^2 &= |L(e_1)|^2 - |L_1(e_1)|^2 \le \norm{L}^2 - \norm{L_1}^2 
\le \left( (1+\varepsilon)^2 - 1\right) \norm{L_1}^2 \\
&= \varepsilon(2+\varepsilon) \norm{L_1}^2 \le 3\varepsilon \norm{L_1}^2.
\end{align*}
By Hadamard's inequality,
\begin{align*}
\star (L')^*\vol_{(\R^n)^{k-1}} &= \sum_{j=2}^k \star L_j^*\vol_{\R^n} \le \sum_{j=2}^k |L_j(e_1)|\cdots |L_j(e_n)| \\
&\le (k-1) \sqrt{3\varepsilon} \norm{L_1} \norm{L'}^{n-1} \le (k-1) \sqrt{3\varepsilon} \norm{L}^n.
\end{align*}
Since 
\[
\star L^*\vol_{(\R^n)^k}^\times = \star L_1^*\vol_{\R^n} + \star (L')^*\vol_{(\R^n)^{k-1}}^\times \le \star L_1^*\vol_{\R^n} + (k-1) \sqrt{3\varepsilon} \norm{L}^n
\]
and $(1+\varepsilon)(k-1) \le k$, we have that
\[
\norm{L_1}^n \le \norm{L}^n \le \frac{1+\varepsilon}{1-k\sqrt{3\varepsilon}} (\star L_1^*\vol_{\R^n}) 
\le (1+7k\sqrt{\varepsilon}) \det L_1.
\]
This proves \eqref{eq:qc-index}.

Finally, to prove \eqref{item:good-index-2}, let $i\in \{2,\ldots, k\}$ and let $v\in \R^n$ be a unit vector for which $\norm{L_i}=|L_i(v)|$. Since $v$ is a unit vector, we have that $|L_1(v)|$ is at least the smallest singular value of $L_1$. 
By \eqref{eq:qc-index},
\[
|L_1(v)|\ge \frac{\det L_1}{\norm{L_1}^{n-1}} = \frac{\det L_1}{\norm{L_1}^n} \norm{L_1} \ge \frac{1}{1+7k\sqrt{\varepsilon}} \norm{L_1} \ge (1-7k\sqrt{\varepsilon}) \norm{L_1},
\]
and, using \eqref{item:good-index},
\begin{align*}
\norm{L_i}^2 &= |L_i(v)|^2 \le |L(v)|^2-|L_1(v)|^2 \le \norm{L}^2 - (1-7k\sqrt{\varepsilon})^2 \norm{L_1}^2 \\
&\le \left( (1+\varepsilon)^2 - (1-7k\sqrt{\varepsilon})^2\right)\norm{L_1}^2
\le 18k \sqrt{\varepsilon} \norm{L_1}^2 
< 25 k \sqrt{\varepsilon} \norm{L}^2.
\end{align*}
This concludes the proof.
\end{proof}

In the following lemma, we record a quantitative estimate for the fact that, linear maps $\R^n \to (\R^n)^k$ close to conformal linear maps $L \colon \R^n \to (\R^n)^k$ satisfying $\norm{L}^n = \star L^*\vol_{(\R^n)^k}^\times$ have small distortion.

\begin{lemma}
\label{lem:conformal-vertices-distortion}
Let $n\ge 2$, $k\ge 1$, and $0<\nu <2^{-n}$. Let $L\colon \R^n \to (\R^n)^k$ be a linear map for which $\star L^*\vol_{(\R^n)^k}^\times =\norm{L}^n$. Let $R\colon \R^n \to (\R^n)^k$ be a linear map satisfying
\[
|R(e_i)-L(e_i)| \le \frac{\nu}{\sqrt{n}} \norm{L}
\]
for $i=1,\ldots,n$. Then $\norm{R}\le (1+\nu)\norm{L}$ and
\[
\norm{R}^n \le \frac{(1+\nu)^n}{1-2^n \nu} \left( \star R^*\vol_{(\R^n)^k}^\times \right).
\]
\end{lemma}

\begin{proof}
On one hand, let $v=(v_1,\ldots,v_n) \in \R^n$ for which $|v|=1$. Then
\begin{align*}
|R(v)| &\le |L(v)| + |R(v)-L(v)| \le \norm{L} + \sum_{i=1}^n |v_i| |R(e_i)-L(e_i)| \\
&\le \norm{L} + \frac{\nu}{\sqrt{n}} \norm{L} \sum_{i=1}^n |v_i| \le \norm{L} + \frac{\nu}{\sqrt{n}} \norm{L} \sqrt{n} |v| \\
&= (1+\nu)\norm{L}.
\end{align*}
Thus $\norm{R}\le (1+\nu)\norm{L}$.

On the other hand, let $x_i=R(e_i)-L(e_i)$ for $i=1,\ldots,n$. Then
\begin{align*}
\star R^*\vol_{(\R^n)^k}^\times &= \vol_{(\R^n)^k}^\times(L(e_1)+x_1,\ldots,L(e_n)+x_n) \\
&= \star L^*\vol_{(\R^n)^k}^\times + \sum_{j=1}^{2^n-1} \vol_{(\R^n)^k}^\times(y_{j_1},\ldots,y_{j_n}),
\end{align*}
where each $y_{j_i}$ is either $y_{j_i}=L(e_i)$ or $y_{j_i}=x_i$ and each $n$-tuple $(y_{j_1},\ldots,y_{j_n})$ contains at least one $y_{j_i}=x_i$. Thus, for each $j\in \{1,\ldots, 2^n-1\}$, we have by Hadamard's inequality that
\[
\left| \vol_{(\R^n)^k}^\times(y_{j_1},\ldots,y_{j_n}) \right| \le \prod_{i=1}^n |y_{j_i}| \le \left( \frac{\nu}{\sqrt{n}} \norm{L} \right) (\norm{L}^{n-1}) \le \nu \norm{L}^n.
\]
We obtain that
\[
\star R^*\vol_{(\R^n)^k}^\times \ge (1-2^n \nu)\norm{L}^n
\]
and the claim follows.
\end{proof}


\section{Approximation of quasiregular $\vol_{(\R^n)^k}^\times$-curves of small distortion with piecewise linear curves}
\label{sec:approximation}

In this section we prove a version of Kopylov's approximation theorem for quasiregular $\vol_{(\R^n)^k}^\times$-curves. Theorem \ref{thm:4}, stated in the introduction, is a direct consequence of Theorem \ref{thm:main-approx} below. Since the piecewise linear approximation is given almost canonically, we begin by restating the result. 

Let $n\ge 3$ and let $E\subset \R^n$ be a compact set. For each $j\in \N$, let $\mathcal Q_j(E)$ be the collection of dyadic cubes $Q$ with side length $2^{-j}$ for which the intersection $Q\cap E$ is non-empty. Here $Q\subset \R^n$ is a dyadic cube with side length $2^{-j}$ if $Q=2^{-j}(v+[0,1]^n)$ for some $v\in \Z^n$. Further, let $\Delta_j(E)$ be the collection of $n$-simplicies obtained by taking barycentric subdivisions of cubes in $\mathcal Q_j(E)$ and let
\[
D_j(E) = \bigcup_{\Delta \in \Delta_j(E)} \Delta \subset \R^n.
\]

Given a map $F\colon \Omega \to (\R^n)^k$, we denote $\widehat F_j^E$ the piecewise affine map $\widehat F_j^E \colon D_j(E) \to (\R^n)^k$ satisfying
\[
\widehat F_j^E(v_i) = F(v_i)
\]
for $i=0,\ldots,n$ whenever the $n$-simplex $[v_0,\ldots,v_n]$ belongs to $\Delta_j(E)$. We say that $\widehat F_j^E$ is the \emph{piecewise affine map associated to $F$ on $\Delta_j(E)$}.

\begin{theorem}
\label{thm:main-approx}
Let $n\ge 3$, $k\ge 1$, and $\varepsilon >0$. There exists $\delta=\delta(n,k,\varepsilon)>0$ for the following. Let $\Omega \subset \R^n$ be a domain, let $U\Subset \Omega$ be a compactly contained subdomain, and let $F\colon \Omega\to (\R^n)^k$ be a $(1+\delta)$-quasiregular $\vol_{(\R^n)^k}^\times$-curve. Then there exists a sequence $(j_m)$ so that for each $m\in \N$ and $j\ge j_m$, the piecewise affine map $(\widehat F_j^{\overline U})|_{U} \colon U \to (\R^n)^k$ is a $(1+\varepsilon)$-quasiregular $\vol_{(\R^n)^k}^\times$-curve satisfying
\[
\sup_U |F-\widehat F_j^{\overline U}| < \frac{1}{m}.
\]
\end{theorem}

We divide the proof of Theorem \ref{thm:main-approx} into three parts. First, we recall classical estimates for M\"obius transformations on $\R^n$ with respect to their derivatives. Second, we prove results regarding affine maps agreeing with a quasiregular $\vol_{(\R^n)^k}^\times$-curve of small distortion at the vertices of a simplex. Finally, we prove Theorem \ref{thm:main-approx} using these results.

\subsection{M\"obius estimates}

We begin by recalling two classical estimates for M\"obius transformations. As they have a crucial role in the proof of Theorem \ref{thm:main-approx}, we give short proofs. The first lemma is a distortion estimate in terms of the derivative.

\begin{lemma}
\label{lem:mob-bounds}
Let $n\ge 3$ and let $\varphi \colon B^n \to \R^n$ be a restriction of a M\"obius transformation $\bSn \to \bSn$. Then
\begin{equation}
\label{eq:mob-lower-upper}
\frac{1}{3} \norm{D\varphi(0)} \le \sup_{x\in B^n(\frac{1}{2})} |\varphi(x)-\varphi(0)| \le \norm{D\varphi(0)}.
\end{equation}
\end{lemma}

\begin{proof}
It suffices to consider the case that $\varphi$ is a restriction of a M\"obius transformation, also denoted $\varphi$, satisfying $\varphi^{-1}(\infty)\ne \infty$.

Let $(u_m)$ be a sequence in $B^n$ tending to the origin. Then
\[
\lim_{m\to \infty} \frac{|\varphi(u_m)-\varphi(0)|}{|u_m|} = \norm{D\varphi(0)}.
\]

Let $x\in B^n(\frac{1}{2})$. Then, for each $m$, we have
\[
|x-p|\frac{|\varphi(x)-\varphi(0)|}{|x|}=|u_m-p|\frac{|\varphi(u_m)-\varphi(0)|}{|u_m|},
\]
where $p\in \R^n \setminus B^n$ is the pole of $\varphi$; see e.g. \cite[(7)]{Iwaniec-PAMS-1987}. 

To obtain the second inequality in \eqref{eq:mob-lower-upper}, it suffices to observe that
\[
|\varphi(x)-\varphi(0)|=\frac{|x||p|}{|x-p|} \norm{D\varphi(0)} \le \norm{D\varphi(0)}.
\]

To obtain the first inequality in \eqref{eq:mob-lower-upper}, it suffices to observe that
\[
|\varphi(p/(2|p|))-\varphi(0)| = \frac{1}{2} \frac{|p|}{|p/(2|p|)-p|} \norm{D\varphi(0)} \ge \frac{1}{3} \norm{D\varphi(0)}.
\]
\end{proof}

In the following lemma, we record the fact that, for a M\"obius transformation $\varphi \colon B^n\to \R^n$, the second order error term $x\mapsto |\varphi(x)-\varphi(0)-D\varphi(0)x|$ is controlled by $x\mapsto \norm{D\varphi(0)} |x|^2$ near the origin.

\begin{lemma}
\label{lem:mob-second-order}
Let $n\ge 3$ and let $\varphi \colon B^n \to \R^n$ be a restriction of a M\"obius transformation $\bSn \to \bSn$. Then there exists $C=C(n)>0$ for which
\[
|\varphi(x)-\varphi(0)-D\varphi(0)x| \le C\norm{D\varphi(0)} |x|^2
\]
for $x\in B^n(\frac{1}{2})$.
\end{lemma}

\begin{proof}
Let $x\in B^n(\frac{1}{2})$. Then there exists $\xi_1,\ldots,\xi_n \in B^n(\frac{1}{2})$ satisfying
\[
\varphi(x) = \varphi(0) + D\varphi(0)x + \frac{1}{2} (x^t H_1 x,\ldots, x^t H_n x),
\]
where $H_i=(h_{\ell j}^i)=D^2 \varphi_i (\xi_i)$ for each $i$. We obtain
\[
|\varphi(x)-\varphi(0)-D\varphi(0)x| \le \frac{1}{2} \sum_{i=1}^n |x^t H_i x| \le \frac{1}{2} |x|^2 \sum_{i=1}^n \norm{H_i}
\]
by the Cauchy--Schwartz inequality. For each $i$, we have that $\norm{H_i} \le C(n) \max_{\ell,j} |h_{\ell j}^i|$, so it suffices to show that
\[
\max_{\ell,j} |h_{\ell j}^i| \le C(n) \norm{D\varphi(0)}.
\]

Since $\varphi$ is a M\"obius transformation on $\R^n$ and the pole of $\varphi$ is outside the unit ball $B^n$, there exists an orthogonal matrix $A=(a_{ij}) \in \R^{n\times n}$, $b\in \R^n$, $p=(p_1,\ldots,p_n) \in \R^n \setminus B$, $c>0$, and $\zeta \in \{0,2\}$ satisfying
\[
\varphi(y) = b + \frac{cA(y-p)}{|y-p|^\zeta},
\]
for $y\in B^n$; see e.g. \cite{Gehring-Martin-Palka-book} and  \cite{Iwaniec-Martin-book}. If $\zeta=0$, the claim is immediate, so suppose that $\zeta=2$. Then we also have that $\norm{D\varphi(0)} = \frac{c}{|p|^2}$.

Now, for $i,j=1,\ldots,n$ and for $y=(y_1,\ldots,y_n) \in B^n$, we have
\[
\frac{\partial \varphi_i}{\partial y_j}(y) = \frac{ca_{ij}}{|y-p|^2} - \frac{2c(y_j-p_j)}{|y-p|^4} s_i(y),
\]
where $s_i(y)=\sum_{m=1}^n a_{im}(y_m-p_m)$. Thus, for $i,j,\ell=1,\ldots,n$ and $y=(y_1,\ldots,y_n) \in B^n$, we have that
\begin{align*}
\frac{\partial^2 \varphi_i}{\partial y_\ell \partial y_j}(y) =& ca_{ij} \frac{\partial}{\partial y_\ell} \left( \frac{1}{|y-p|^2} \right) - 2c \frac{\partial}{\partial y_\ell} \left( \frac{(y_j-p_j)}{|y-p|^4} s_i(y) \right) \\
=& - 2ca_{ij}\frac{y_\ell-p_\ell}{|y-p|^4} 
-\frac{2c}{|y-p|^4}\left( \delta_{\ell j}s_i(y) + a_{i\ell}(y_j-p_j) \right) \\
&+\frac{2c}{|y-p|^4} \frac{4(y_j-p_j)(y_\ell-p_\ell)}{|y-p|^2}s_i(y),
\end{align*}
where $\delta_{\ell j}$ is the Kronecker delta.

Since $D^2 \varphi_i (\xi_i)=(h_{\ell j}^i)$ and $\xi_i \in B^n(\frac{1}{2})$ for each $i$, it follows that we have
\[
|h_{\ell j}^i| \le C(n)\frac{c}{|\xi_i-p|^3}
\]
for each $i$, $j$, and $\ell$. On the other hand, since $\frac{1}{2}\le \frac{1}{2}|p|\le |\xi_i-p|$ for each $i$, we have that
\[
\frac{c}{|\xi_i-p|^3} = \frac{1}{|\xi_i-p|} \frac{c}{|\xi_i-p|^2} \le 2\cdot 4 \frac{c}{|p|^2} = 8 \norm{D\varphi(0)}.
\]
This concludes the proof.
\end{proof}

\subsection{Affine approximation of a quasiregular $\vol_{(\R^n)^k}^\times$-curve in a simplex}

The following lemma is similar to \cite[Lemma 6]{Kopylov-1972}. It states that for each $\varepsilon>0$ there exists $\delta>0$ for which a $(1+\delta)$-quasiregular $\vol_{(\R^n)^k}^\times$-curve $F\colon B^n(2)\to (\R^n)^k$, normalized to satisfy $F(0)=0$ and $\sup_{B^n}|F| =1$, can be approximated near the origin by a linear map $P\colon \R^n \to (\R^n)^k$, which is also a $(1+\varepsilon)$-quasiregular $\vol_{(\R^n)^k}^\times$-curve. In this case, we also obtain a uniform bound for the operator norm of $P$, which will be useful later.

\begin{lemma}
\label{lem:pl-approx-model}
Let $n\ge 3$, $k\ge 1$, and $\varepsilon >0$. There exists $t_0=t_0(n,\varepsilon)\in (0,1/2)$ and $\delta=\delta(n,k,\varepsilon)>0$ for the following. Let $F\colon B^n(2) \to (\R^n)^k$ be a $(1+\delta)$-quasiregular $\vol_{(\R^n)^k}^\times$-curve satisfying $F(0)=0$ and $\sup_{B^n} |F|=1$. Let $P\colon \R^n \to (\R^n)^k$ be the linear map defined by $P(e_i)=\frac{1}{t_0}F(t_0e_i)$ for $i=1,\ldots,n$. Then $P$ is a $(1+\varepsilon)$-quasiregular $\vol_{(\R^n)^k}^\times$-curve and $\norm{P} \le 12$.
\end{lemma}

\begin{proof}
Let $C_0=C_0(n)>0$ be as in Lemma \ref{lem:mob-second-order}. Let $0<\nu<2^{-n}$ for which
\[
\frac{(1+\nu)^n}{1-2^n \nu} < 1+\varepsilon
\]
and let $t_0\in (0,1/2)$ satisfy $2^{13}\cdot 7C_0t_0\le \frac{\nu}{\sqrt{n}}$. Let then $\delta >0$ for which $\tau^{n,k}(\delta)\le \min (2^{-14},C_0t_0^2/2)$, where $\tau^{n,k}$ is the gauge function defined in Section \ref{sec:Iwaniec}.

Let $F\colon B^n(2) \to (\R^n)^k$ be a $(1+\delta)$-quasiregular $\vol_{(\R^n)^k}^\times$-curve satisfying $F(0)=0$ and $\sup_{B^n} |F|=1$. Let $P\colon \R^n \to (\R^n)^k$ be the linear map which is defined by $P(e_i)=\frac{1}{t_0}F(t_0e_i)$ for $i=1,\ldots,n$.

By Proposition \ref{prop:1-qr-near}, there exists a map $\varphi \in \ccF^{n,k}(B^n)$ for which $\varphi(0)=0$ and
\[
\sup_{B^n(\frac{1}{2})} |F-\varphi| \le 2\tau^{n,k}(\delta).
\]
Now $\varphi=(\varphi_1,\ldots,\varphi_k) \colon B^n \to (\R^n)^k$ is a map for which either $\varphi \equiv 0$ or $\varphi=(0,\ldots,0,\varphi_{i_0},0,\ldots,0)$, where $\varphi_{i_0} \colon B^n \to \R^n$ is a restriction of a M\"obius transformation $\bSn \to \bSn$. It follows that we have
\[
\norm{D\varphi(0)}^n = \star (D\varphi(0))^* \vol_{(\R^n)^k}^\times.
\]
Then, by Lemma \ref{lem:conformal-vertices-distortion}, it suffices to show that
\[
|P(e_i)-(D\varphi(0))(e_i)| \le \frac{\nu}{\sqrt{n}} \norm{D\varphi(0)}
\]
for $i=1,\ldots,n$ to obtain that $P$ is a $(1+\varepsilon)$-quasiregular $\vol_{(\R^n)^k}^\times$-curve and
\[
\norm{P} \le (1+\nu) \norm{D\varphi(0)} \le 2\norm{D\varphi(0)}.
\]

Since $\tau^{n,k}(\delta)\le 10^{-3}$, we obtain that
\[
\sup_{B^n(\frac{1}{2})} |F| \ge 2^{-12}
\]
by applying Proposition \ref{prop-iwaniecstabilitylemma3} in the ball $B^n(2)$ with $\rho=\frac{1}{4}$. Further, the estimate $\tau^{n,k}(\delta)\le 2^{-14}$ yields that
\[
\sup_{B^n(\frac{1}{2})} |\varphi| \le 2\tau^{n,k}(\delta) + 1 \le 2
\]
and
\[
\sup_{B^n(\frac{1}{2})} |\varphi| \ge 2^{-12} - 2\tau^{n,k}(\delta) \ge 2^{-13}.
\]
We obtain that there exists a unique M\"obius transformation coordinate map $\varphi_{i_0} \colon B^n \to \R^n$ satisfying
\[
2^{-13} \le \sup_{B^n(\frac{1}{2})} |\varphi_{i_0}| \le 2.
\]
Then \eqref{eq:mob-lower-upper} yields the estimate
\[
2^{-13}\le \norm{D\varphi_{i_0}(0)}\le 6.
\]

Now, for each $i$, we have
\begin{align*}
|P(e_i)-(D\varphi(0))(e_i)| &= \frac{1}{t_0}|F(t_0e_i)-(D\varphi(0))(t_0e_i)| \\
&\le \frac{1}{t_0} \left( |F(t_0e_i)-\varphi(t_0e_i)| + |\varphi(t_0e_i)-(D\varphi(0))(t_0e_i)| \right) \\
&\le \frac{1}{t_0} \left( 2\tau^{n,k}(\delta) + |\varphi_{i_0}(t_0e_i)-(D\varphi_{i_0}(0))(t_0e_i)| \right).
\end{align*}
Further, by Lemma \ref{lem:mob-second-order}, we have that
\[
|\varphi_{i_0}(t_0e_i)-(D\varphi_{i_0}(0))(t_0e_i)| \le C_0 \norm{D\varphi_{i_0}(0)} t_0^2 \le 6C_0 t_0^2.
\]
Combining these estimates, we obtain
\begin{align*}
|P(e_i)-(D\varphi(0))(e_i)| &\le 2^{13}\frac{2\tau^{n,k}(\delta)+6C_0 t_0^2}{t_0} \norm{D\varphi(0)} \\
&\le 2^{13}\frac{C_0t_0^2+6C_0t_0^2}{t_0} \norm{D\varphi(0)} = 2^{13}\cdot 7C_0 t_0 \norm{D\varphi(0)} \\
&\le \frac{\nu}{\sqrt{n}} \norm{D\varphi(0)},
\end{align*}
since $\tau^{n,k}(\delta) \le C_0t_0^2/2$ and $2^{13}\cdot 7C_0 t_0 \le \frac{\nu}{\sqrt{n}}$. This concludes the proof.
\end{proof}

We obtain the following proposition from Lemma \ref{lem:pl-approx-model} by scaling.

\begin{proposition}
\label{prop:pa-stand-simplex}
Let $n\ge 3$, $k\ge 1$, and $\varepsilon >0$. Let $t_0=t_0(n,\varepsilon)\in (0,1/2)$ and $\delta=\delta(n,k,\varepsilon)>0$ be as in Lemma \ref{lem:pl-approx-model}. Let $F\colon B^n(2/t_0) \to (\R^n)^k$ be a $(1+\delta)$-quasiregular $\vol_{(\R^n)^k}^\times$-curve and let $P\colon \R^n \to (\R^n)^k$ be the affine map defined by $P(0)=F(0)$ and $P(e_i)=F(e_i)$ for $i=1,\ldots,n$. Then $P$ is a $(1+\varepsilon)$-quasiregular $\vol_{(\R^n)^k}^\times$-curve and
\[
\sup_{x\in \Delta} |P(0)-P(x)| \le 12t_0 \sup_{x\in B^n(1/t_0)} |F(0)-F(x)|,
\]
where $\Delta=[0,e_1,\ldots,e_n]\subset \R^n$ is the standard $n$-simplex.
\end{proposition}

\begin{proof}
We may assume that
\[
m = \sup_{x\in B^n(1/t_0)} |F(0)-F(x)| > 0.
\]
Let $\Psi \colon B^n(2)\to (\R^n)^k$ be the map $x\mapsto \frac{1}{m}(F(x/t_0)-F(0))$. Then, the map $\Psi$ is a $(1+\delta)$-quasiregular $\vol_{(\R^n)^k}^\times$-curve for which $\Psi(0)=0$ and $\sup_{B^n} |\Psi|=1$. Let $\tilde P \colon \R^n \to (\R^n)^k$ be the linear map defined by $\tilde P(e_i)=\frac{1}{t_0}\Psi(t_0e_i)$ for $i=1,\ldots,n$.

Now, by Lemma \ref{lem:pl-approx-model}, $\tilde P$ is a $(1+\varepsilon)$-quasiregular $\vol_{(\R^n)^k}^\times$-curve and $\norm{\tilde P}\le 12$. Since $P = F(0) + mt_0\tilde P$, we obtain that $P$ is a $(1+\varepsilon)$-quasiregular $\vol_{(\R^n)^k}^\times$-curve and
\[
\sup_{x\in \Delta} |P(0)-P(x)| = mt_0 \sup_{x\in \Delta} |\tilde P(x)| \le mt_0 \norm{\tilde P} \sup_{x\in \Delta} |x| \le 12mt_0.
\]
\end{proof}

Having Proposition \ref{prop:pa-stand-simplex} at our disposal, we are ready to prove Theorem \ref{thm:main-approx}.

\begin{proof}[Proof of Theorem \ref{thm:main-approx}]
Let $n\ge 3$, $k\ge 1$, and $\varepsilon >0$. Let $t_0=t_0(n,\varepsilon)\in (0,1/2)$ and $\delta=\delta(n,k,\varepsilon)>0$ be as in Lemma \ref{lem:pl-approx-model}. Let $\Omega \subset \R^n$ be a domain, let $U\Subset \Omega$ be a compactly contained subdomain, and let $F\colon \Omega\to (\R^n)^k$ be a $(1+\delta)$-quasiregular $\vol_{(\R^n)^k}^\times$-curve. Let also $m\in \N$.

Let $V\Subset \Omega$ be a compactly contained subdomain satisfying $\overline U \subset V$ and let $\eta_m >0$ for which $\eta_m(1+12t_0)<\frac{1}{m}$. Since $F$ is uniformly continuous on $\overline V$, there exists $\theta_m >0$ for which $|F(x)-F(y)|<\eta_m$ for $x,y\in V$ satisfying $|x-y|<\theta_m$. Fix $j_m \in \N$ having the property that $\frac{1}{t_0} 2^{-(j_m+1)}<\theta_m$ and, for every $Q\in \mathcal Q_{j_m}(\overline U)$, we have $\diam Q< \theta_m$ and $\dist (Q,\partial V)>\frac{1}{t_0} 2^{-j_m}$. Let now $j\ge j_m$ and let $\widehat F_j^{\overline U} \colon D_j(\overline U) \to (\R^n)^k$ be the piecewise affine map associated to $F$ on $\Delta_j(\overline U)$.

Let $\Delta \in \Delta_j(\overline U)$. Since $\Delta$ is an $n$-simplex, there exists $v_0,\ldots,v_n \in \R^n$ for which $\Delta=[v_0,\ldots,v_n]$.  Let $P\colon \R^n \to (\R^n)^k$ be the affine map defined by $P(v_i)=F(v_i)$ for $i=0,\ldots,n$. Then $P|_{\Delta}=(\widehat F_j^{\overline U})|_{\Delta}$ and it suffices to show that $P$ is a $(1+\varepsilon)$-quasiregular $\vol_{(\R^n)^k}^\times$-curve which is uniformly close to $F$ in $\Delta$.

Since the $n$-simplex $[v_0,\ldots,v_n]$ is obtained from the barycentric subdivision of some cube $Q\in \mathcal Q_j(\overline U)$, there exists an index $i_0 \in \{0,\ldots,n\}$ and a permutation $\sigma \colon (\{0,\ldots,n\} \setminus \{i_0\}) \to \{1,\ldots,n\}$ with the property that $v_i = v_{i_0} \pm 2^{-(j+1)}e_{\sigma(i)}$ for $i\ne i_0$. Let $B=B^n(v_{i_0},2^{-(j+1)}/t_0)\subset \R^n$ be a ball. Since $v_{i_0}\in Q$ and $\dist (Q,\partial V)>\frac{1}{t_0} 2^{-j}$, we have that $2B\subset V$.

Let $A\colon \R^n \to \R^n$ be the affine map defined by $A(0)=v_{i_0}$ and $A(e_{\sigma(i)})=v_i$ for $i\ne i_0$. Then applying Proposition \ref{prop:pa-stand-simplex} to the map $F\circ A$ yields that $P$ is a $(1+\varepsilon)$-quasiregular $\vol_{(\R^n)^k}^\times$-curve for which
\[
\sup_{x\in \Delta} |P(v_{i_0})-P(x)|\le 12t_0 \sup_{x\in B} |F(v_{i_0})-F(x)|.
\]

Now, since $\Delta \cup B \subset V$ and $|x-v_{i_0}|\le \max(\diam Q,2^{-(j+1)}/t_0)<\theta_m$ for $x\in \Delta \cup B$, we have that
\begin{align*}
\sup_{x\in \Delta} |F(x)-P(x)| &\le \sup_{x\in \Delta} |F(x)-F(v_{i_0})| + \sup_{x\in \Delta} |P(v_{i_0})-P(x)| \\
&\le \sup_{x\in \Delta} |F(x)-F(v_{i_0})| + 12t_0 \sup_{x\in B} |F(v_{i_0})-F(x)| \\
&\le \eta_m(1+12t_0) < \frac{1}{m}.
\end{align*}
This concludes the proof.
\end{proof}


\section{Quasiregular $\vol_{(\R^n)^n}^\times$-curves of small distortion have a quasiregular dominating coordinate map}

In this section, we prove the following theorem, which is the Euclidean and quantitative counterpart of Theorem \ref{thm:stability-main}.

\begin{theorem}
\label{thm:Euclidean-stability-qr-curves}
Let $n\ge 3$, $k\ge 1$ and let $\varepsilon >0$. There exists $\delta=\delta(n,k,\varepsilon)>0$ for the following. Let $\Omega \subset \R^n$ be a domain and let $F=(f_1,\ldots,f_k) \colon \Omega \to (\R^n)^k$ be a non-constant $(1+\delta)$-quasiregular $\vol_{(\R^n)^k}^\times$-curve. Then there exists a unique index $i_0 \in \{1,\ldots,n\}$ for which the coordinate map $f_{i_0} \colon \Omega \to \R^n$ is non-constant and $(1+7k\sqrt{\varepsilon})$-quasiregular. Moreover, for almost every $x\in \Omega$, we have that
\[
\norm{DF(x)} \le (1+\varepsilon)(1+7k\sqrt{\varepsilon})^\frac{1}{n} \norm{Df_{i_0}(x)}
\]
and
\[
\norm{Df_i(x)} \le 5\sqrt{k}\varepsilon^\frac{1}{4}(1+\varepsilon)^\frac{1}{n} \norm{DF(x)}
\]
for $i\ne i_0$.
\end{theorem}

First, we prove a result similar to Proposition \ref{prop:linear-1+epsilon} for piecewise affine maps.

\begin{proposition}
\label{prop:pl-unique-coordinate}
Let $n\ge 3$ and $k\ge 1$. There exists $\varepsilon=\varepsilon(n,k)>0$ for the following. Let $\Delta_1$ and $\Delta_2$ be $n$-simplicies in $\R^n$ for which the intersection $\Delta=\Delta_1 \cap \Delta_2$ is an $(n-1)$-simplex. Let $P=(P_1,\ldots,P_k) \colon (\Delta_1 \cup \Delta_2) \to (\R^n)^k$ be a piecewise affine map for which $P_{|\Delta_j}$ is non-constant for $j=1,2$ and
\[
\norm{DP}^n \le (1+\varepsilon) \left( \star P^*\vol_{(\R^n)^k}^\times \right)
\]
in $(\mathrm{int} (\Delta_1 \cup \Delta_2))\setminus \Delta$. Then there exists $i_0 \in \{1,\ldots,k\}$ for which $P_{i_0}$ is $(1+7k\sqrt{\varepsilon})$-quasiregular in $\mathrm{int} (\Delta_1 \cup \Delta_2)$. Furthermore, for $i\ne i_0$, we have $\norm{DP_{i_0}} \ge \norm{DP}/(1+\varepsilon)$ and $\norm{DP_i} \le 5\sqrt{k}\varepsilon^\frac{1}{4} \norm{DP}$ in $(\mathrm{int} (\Delta_1 \cup \Delta_2))\setminus \Delta$.
\end{proposition}

\begin{proof}
Let $\varepsilon \in (0,1/(100k))$ satisfy $(1+\varepsilon)(1+7k\sqrt{\varepsilon})<3/2$ and $5\sqrt{k}\varepsilon^\frac{1}{4}<1/2$.

Let $\Delta_1$ and $\Delta_2$ be $n$-simplicies in $\R^n$ for which the intersection $\Delta=\Delta_1 \cap \Delta_2$ is an $(n-1)$-simplex and let $P=(P_1,\ldots,P_k) \colon (\Delta_1 \cup \Delta_2) \to (\R^n)^k$ be a piecewise affine map for which $P_{|\Delta_j}$ is non-constant for $j=1,2$ and
\[
\norm{DP}^n \le (1+\varepsilon) \left( \star P^*\vol_{(\R^n)^k}^\times \right)
\]
in $(\mathrm{int} (\Delta_1 \cup \Delta_2))\setminus \Delta$. By translation, we may assume that $0\in \Delta$ and that $P(0)=0$.

Since $P$ is linear and non-constant in $\Delta_1$, there exists a non-trivial linear map $L=(L_1,\ldots,L_k) \colon \R^n \to (\R^n)^k$ for which $P(x)=L(x)$ for $x\in \Delta_1$. Similarly, there exists a non-trivial linear map $R=(R_1,\ldots,R_k) \colon \R^n \to (\R^n)^k$ satisfying $P(x)=R(x)$ for $x\in \Delta_2$. Since $\varepsilon < 1/(100k)$, we obtain indices $i_L,i_R \in \{1,\ldots,k\}$ as in in Proposition \ref{prop:linear-1+epsilon} for $L$ and $R$, respectively. It suffices to show that $i_L=i_R$.

Suppose towards contradiction that $i_L \ne i_R$. Let $V\subset \R^n$ be the linear subspace spanned by $\Delta$ and let $\sigma_1 \le \cdots \le \sigma_n$ be the singular values of $L_{i_L}$. Then, $\norm{L_{i_L}}=\sigma_n$ and $\norm{(L_{i_L})|_{V}}\ge \sigma_1$. It follows that
\begin{align*}
\norm{L_{i_L}}^n &\le (1+7k\sqrt{\varepsilon}) \det L_{i_L} \le (1+7k\sqrt{\varepsilon}) \prod_{j=1}^n \sigma_j \le (1+7k\sqrt{\varepsilon}) \sigma_1 \sigma_n^{n-1} \\
&\le (1+7k\sqrt{\varepsilon}) \norm{(L_{i_L})|_{V}} \norm{L_{i_L}}^{n-1}.
\end{align*}
Thus,
\[
\norm{L} \le (1+\varepsilon)\norm{L_{i_L}} \le (1+\varepsilon)(1+7k\sqrt{\varepsilon}) \norm{(L_{i_L})|_{V}} \le \frac{3}{2} \norm{(L_{i_L})|_{V}}.
\]
On the other hand, since $i_L \ne i_R$, we have
\[
\norm{(R_{i_L})|_{V}} \le \norm{R_{i_L}} \le 5\sqrt{k}\varepsilon^\frac{1}{4} \norm{R} \le \frac{1}{2} \norm{R}.
\]

Since $L|_{\Delta}=R|_{\Delta}$ and both maps $L$ and $R$ are linear, we have that $L|_{V}=R|_{V}$. Then, combining the obtained estimates yields
\[
\norm{L} \le \frac{3}{2} \norm{(L_{i_L})|_{V}} = \frac{3}{2} \norm{(R_{i_L})|_{V}} \le \frac{3}{4} \norm{R}.
\]
Since $\norm{R}>0$, this implies that $\norm{L}<\norm{R}$. Similarly, we obtain that $\norm{R}<\norm{L}$. This contradiction yields the claim.
\end{proof}

Now we are ready to prove Theorem \ref{thm:Euclidean-stability-qr-curves}.

\begin{proof}[Proof of Theorem \ref{thm:Euclidean-stability-qr-curves}]
Let $c_1=(1+\varepsilon)(1+7k\sqrt{\varepsilon})^\frac{1}{n}$ and $c_2=5\sqrt{k}\varepsilon^\frac{1}{4}(1+\varepsilon)^\frac{1}{n}$. By passing to a smaller $\varepsilon$ if necessary, we may assume that Proposition \ref{prop:pl-unique-coordinate} holds for $\varepsilon$ and that $c_1 c_2 <1$. Let $\delta_1=\delta_1(n,k,\varepsilon)>0$ be as in Theorem \ref{thm:4}. By Theorem \ref{thm:local-injectivity}, we may fix $\delta_2=\delta_2(n,k)>0$ for which every non-constant $(1+\delta_2)$-quasiregular $\vol_{(\R^n)^k}^\times$-curve is a local embedding.

Let now $0<\delta \le \min(\delta_1,\delta_2)$, let $\Omega \subset \R^n$ be domain, and let $F=(f_1,\ldots,f_k) \colon \Omega \to (\R^n)^k$ be a non-constant $(1+\delta)$-quasiregular $\vol_{(\R^n)^k}^\times$-curve. Since $c_1 c_2 <1$, the uniqueness of a suitable index follows immediately, so it suffices to show existence.

Let $U\Subset \Omega$, $V\Subset \Omega$, and $W\Subset \Omega$ be compactly contained subsets satisfying $\overline U\subset V$ and $\overline V\subset W$. Since $F$ is a local embedding and $\overline W\subset \Omega$ is compact, there exists $R>0$ for which $F(x)\ne F(y)$ for $x,y\in W$ satisfying $|x-y|<R$. Fix $j_0 \in \N$ for which $\diam Q< R$ and $Q\subset W$ for every $Q\in \mathcal Q_{j_0}(\overline V)$.

By Theorem \ref{thm:main-approx}, there exists a sequence $(\widehat F_m)$ of piecewise affine $(1+\varepsilon)$-quasiregular $\vol_{(\R^n)^k}^\times$-curves $\widehat F_m=(\widehat f_1^m,\ldots,\widehat f_k^m) \colon V\to (\R^n)^k$ converging uniformly to $F$ in $V$, where, for each $m$, we have $\widehat F_m=(\widehat F_{j_m}^{\overline V})|_{V}$ for some $j_m \ge j_0$.

The restriction $(\widehat F_{j_m}^{\overline V})|_{\Delta}$ is non-constant for each $m$ and for $\Delta \in \Delta_{j_m}(\overline V)$. Thus, by Proposition \ref{prop:pl-unique-coordinate}, for each $m$, there exists an index $i_m\in \{1,\ldots,k\}$ for which the coordinate map $\widehat f_{i_m}^m \colon V \to \R^n$ is $(1+7k\sqrt{\varepsilon})$-quasiregular. Furthermore, for each $m$ and for $i\ne i_m$, we have $\norm{D\widehat f_{i_m}^m} \ge \norm{D\widehat F_m}/(1+\varepsilon)$ and $\norm{D\widehat f_i^m} \le 5\sqrt{k}\varepsilon^\frac{1}{4} \norm{D\widehat F_m}$ almost everywhere in $V$.

By passing to a subsequence if necessary, we may assume that there exists $i_0 \in \{1,\ldots,k\}$ satisfying $i_0=i_m$ for every $m\in \N$. The sequence $(\widehat f_{i_0}^m)$ converges uniformly to $f_{i_0}$ in $V$. Since each $\widehat f_{i_0}^m$ is $(1+7k\sqrt{\varepsilon})$-quasiregular, we obtain that $f_{i_0}$ is $(1+7k\sqrt{\varepsilon})$-quasiregular in $V$; see e.g. \cite[Theorem VI.8.6]{Rickman-book}.

By \cite[Lemma 4.2]{Pankka-AASF-2020} and by passing to a subsequence if necessary, we may assume that $(\widehat F_m)$ converges weakly to $F$ in $W_{\loc}^{1,n}(U,(\R^n)^k)$. This also implies that $(\widehat f_i^m)$ converges weakly to $f_i$ in $W_{\loc}^{1,n}(U,\R^n)$ for $i=1,\ldots,k$.

Let now $x\in U$ be a Lebesgue point of $\norm{DF}^n$ and each $\norm{Df_i}^n$. Let also $B=B^n(x,r)\Subset U$ be a compactly contained ball and fix $i\ne i_0$. Then
\[
\int_B \norm{DF}^n \le \liminf_{m\to \infty} \int_B \norm{D\widehat F_m}^n
\]
and
\[
\int_B \norm{Df_i}^n \le \liminf_{m\to \infty} \int_B \norm{D\widehat f_i^m}^n
\]
by \cite[Lemma 4.4]{Pankka-AASF-2020}. We also have that
\[
\int_B \norm{D\widehat F_m}^n \le (1+\varepsilon)^n \int_B \norm{D\widehat f_{i_0}^m}^n \le (1+\varepsilon)^n (1+7k\sqrt{\varepsilon}) \int_B (\widehat f_{i_0}^m)^* \vol_{\R^n}
\]
and
\[
\int_B \norm{D\widehat f_i^m}^n \le (5\sqrt{k}\varepsilon^\frac{1}{4})^n \int_B \norm{D\widehat F_m}^n \le (5\sqrt{k}\varepsilon^\frac{1}{4})^n (1+\varepsilon) \int_B (\widehat F_m)^* \vol_{(\R^n)^k}^\times
\]
for each $m$.

Let $\zeta \in C_0^\infty(U)$ be a non-negative function satisfying $\zeta \equiv 1$ on $B$. Then
\[
\liminf_{m\to \infty} \int_B (\widehat f_{i_0}^m)^* \vol_{\R^n} \le \liminf_{m\to \infty} \int_U \zeta (\widehat f_{i_0}^m)^* \vol_{\R^n} = \int_U \zeta f_{i_0}^* \vol_{\R^n}
\]
and
\[
\liminf_{m\to \infty} \int_B (\widehat F_m)^* \vol_{(\R^n)^k}^\times \le \liminf_{m\to \infty} \int_U \zeta (\widehat F_m)^* \vol_{(\R^n)^k}^\times = \int_U \zeta F^* \vol_{(\R^n)^k}^\times
\]
by \cite[Lemma 4.3]{Pankka-AASF-2020}. Now combining previous estimates yields
\[
\int_B \norm{DF}^n \le c_1^n \int_U \zeta f_{i_0}^* \vol_{\R^n} \le c_1^n \int_U \zeta \norm{Df_{i_0}}^n
\]
and
\[
\int_B \norm{Df_i}^n \le c_2^n \int_U \zeta F^* \vol_{(\R^n)^k}^\times \le c_2^n \int_U \zeta \norm{DF}^n.
\]

Since $\zeta$ is arbitrary, we obtain that
\[
\int_B \norm{DF}^n \le c_1^n \int_B \norm{Df_{i_0}}^n
\]
and
\[
\int_B \norm{Df_i}^n \le c_2^n \int_B \norm{DF}^n.
\]
Then, since $x$ is a Lebesgue point of $\norm{DF}^n$, $\norm{Df_i}^n$, and $\norm{Df_{i_0}}^n$, letting $r$ tend to zero yields that
\[
\norm{DF(x)}^n \le c_1^n \norm{Df_{i_0}(x)}^n
\]
and
\[
\norm{Df_i(x)}^n \le c_2^n \norm{DF(x)}^n.
\]

By Lebesgue's differentiation theorem, we obtain that $\norm{DF} \le c_1 \norm{Df_{i_0}}$ and $\norm{Df_i} \le c_2 \norm{DF}$ almost everywhere in $U$. Further, since $F$ is a local embedding, we obtain that $f_{i_0}$ is non-constant. The claim follows by exhaustion.
\end{proof}


\section{Proofs of main theorems}
\label{sec:last}

Theorems \ref{thm:local-embedding} and \ref{thm:stability-main} follow from corresponding Euclidean results using charts.

Recall that, given Riemannian manifolds $M$ and $N=N_1 \times \dots \times N_k$, a continuous map $F=(f_1,\ldots,f_k) \colon M\to N$, a point $x\in M$ and a parameter $a\in \N$, the exponential maps at $x$ and $F(x)$ yield charts $(U,\varphi)$ and $(V,\psi)$ on $M$ and $N$, respectively, for which
\begin{enumerate}
\item the set $U$ is a neighborhood of the point $x$,
\item the image $F(U)$ is contained in $V$,
\item the chart $(V,\psi)$ is a product chart, i.e., $V=V_1\times \dots \times V_k$ and $\psi=(\psi_1,\ldots,\psi_k)$, where each $(V_i,\psi_i)$ is a chart on $N_i$,
\item the map $\varphi \colon U\to \R^n$ is $(1+1/a)$-bilipschitz,
\item the map $\psi \colon V\to (\R^n)^k$ is $(1+1/a)$-bilipschitz,
\item the Jacobian determinant of each map $\psi_i$ satisfies $J_{\psi_i}(f_i(x))=1$, and
\item there exists $r_0>0$ for which $\varphi (B_{d_M}(x,r)) = B^n(\varphi(x),r)$ for $0<r<r_0$.
\end{enumerate}
The bilipschitz constant for $\varphi$ and $\psi$ can be made arbitrarily close to one when $U$ and $V$ are chosen to be sufficiently small.
Further we have that
\[
((\psi^{-1})^* \vol_N^\times)(\psi(F(x))=\vol_{(\R^n)^k}^\times.
\]

The following theorem is a quantitative version of Theorem \ref{thm:stability-main}.

\begin{theorem}
\label{thm:stability-qr-curves}
Let $n\ge 3$, $k\ge 1$, and $\varepsilon >0$. Then there exists $\delta_0=\delta_0(n,k,\varepsilon)>0$ for the following. Let $M$ be an oriented and connected Riemannian $n$-manifold and let $N=N_1\times \cdots \times N_k$ be a Riemannian product of oriented Riemannian $n$-manifolds. Then, for $0<\delta<\delta_0$ and a non-constant $(1+\delta)$-quasiregular $\vol_N^\times$-curve $F=(f_1,\ldots, f_k) \colon M\to N$, there exists a unique index $i_0\in \{1,\ldots, k\}$ for which the coordinate map $f_{i_0} \colon M \to N_{i_0}$ is a $(1+8k\sqrt{\varepsilon})$-quasiregular local homeomorphism. Almost everywhere in $M$, we have that
\[
\norm{DF} \le (1+\varepsilon)(1+8k\sqrt{\varepsilon})^\frac{1}{n} \norm{Df_{i_0}}
\]
and
\[
\norm{Df_i} \le 6\sqrt{k}\varepsilon^\frac{1}{4}(1+\varepsilon)^\frac{1}{n} \norm{DF}
\]
for $i\ne i_0$.
\end{theorem}

\begin{proof}
We may assume that $\varepsilon$ is small enough to satisfy
\[
6\sqrt{k}\varepsilon^\frac{1}{4}(1+\varepsilon)^\frac{n+1}{n}(1+8k\sqrt{\varepsilon})^\frac{1}{n}<1
\]
and that each non-constant $(1+8k\sqrt{\varepsilon})$-quasiregular map between Riemannian $n$-manifolds is a local homeomorphism. Let now $\delta_0=\delta_0(n,k,\varepsilon)>0$ be as in Theorem \ref{thm:Euclidean-stability-qr-curves} and let $0<\delta<\delta_0/4$.

Let $M$ be an oriented and connected Riemannian $n$-manifold and let $N=N_1\times \dots \times N_k$ be a Riemannian product of oriented Riemannian $n$-manifolds. Let also $F=(f_1,\ldots,f_k) \colon M\to N$ be a non-constant $(1+\delta)$-quasiregular $\vol_N^\times$-curve.

Since $M$ is connected, it suffices to prove the claim locally. Let $x\in M$ and let $a\in \N$ be an auxiliary parameter to be fixed later. Let $(U,\varphi)$ and $(V,\psi)$ be charts as in the beginning of the section.

We may define a map $G=(g_1,\ldots,g_k)=\psi \circ F\circ \varphi^{-1} \colon \varphi(U) \to \psi(V)$ and an $n$-form $\omega = (\psi^{-1})^* \vol_N^\times \in \Omega^n(\psi(V))$. The $n$-form $\omega$ is a non-vanishing closed form and we have that
\[
(\norm{\omega} \circ G) \norm{DG}^n \le (1+1/a)^{4n}(1+\delta)(\star G^*\omega)
\]
almost everywhere in $\varphi(U)$. We may assume that $(1+1/a)^{4n}(1+\delta)<1+\delta_0/2$. Thus, $G$ is a $K$-quasiregular $\omega$-curve, where $K<1+\delta_0/2$.

By \cite[Lemma 5.2]{Pankka-AASF-2020}, we obtain that the point $\varphi(x)$ has a neighborhood $\Omega \subset \varphi(U)$ for which the restriction $G_{|\Omega} \colon \Omega \to (\R^n)^k$ is a $(1+\delta_0/2)$-quasiregular $\omega_0$-curve, where $\omega_0 \in \wedge^n(\R^n)^k$ satisfies
\[
\omega_0(G(\varphi(x)) = \omega(G(\varphi(x)) = \vol_{(\R^n)^k}^\times.
\]
Thus, the restriction $G_{|\Omega}$ is a non-constant $(1+\delta_0/2)$-quasiregular $\vol_{(\R^n)^k}^\times$-curve. Then, by Theorem \ref{thm:Euclidean-stability-qr-curves}, there exists $i_0\in \{1,\ldots,k\}$ so that the coordinate map $g_{i_0}$ is non-constant and $(1+7k\sqrt{\varepsilon})$-quasiregular in $\Omega$. We also have, almost everywhere in $\Omega$, that $\norm{DG} \le (1+\varepsilon)(1+7k\sqrt{\varepsilon})^\frac{1}{n} \norm{Dg_{i_0}}$ and $\norm{Dg_i} \le 5\sqrt{k}\varepsilon^\frac{1}{4}(1+\varepsilon)^\frac{1}{n} \norm{DG}$ for $i\ne i_0$.

It follows that the coordinate map $f_{i_0}$ is $(1+1/a)^{4n}(1+7k\sqrt{\varepsilon})$-quasiregular in $\varphi^{-1}(\Omega)$. We also obtain, almost everywhere in $\varphi^{-1}(\Omega)$, that
\[
\norm{DF} \le (1+1/a)^4 (1+\varepsilon)(1+7k\sqrt{\varepsilon})^\frac{1}{n} \norm{Df_{i_0}}
\]
and
\[
\norm{Df_i} \le (1+1/a)^4 5\sqrt{k}\varepsilon^\frac{1}{4}(1+\varepsilon)^\frac{1}{n} \norm{DF}
\]
for $i\ne i_0$. The claim now follows by taking $a\in \N$ large enough. 
\end{proof}

Theorem \ref{thm:stability-main} follows from Theorem \ref{thm:stability-qr-curves} immediately. We also obtain Theorem \ref{thm:manifold-Liouville} from Theorem \ref{thm:stability-qr-curves}.

\begin{named}{Theorem \ref{thm:manifold-Liouville}}
Let $M$ be an oriented and connected Riemannian $n$-manifold for $n\ge 3$ and let $N=N_1\times \cdots \times N_k$ be a Riemannian product of oriented Riemannian $n$-manifolds $N_i$ for $i\in \{1,\ldots, k\}$. Then for each non-constant $\vol_N^\times$-calibrated curve $F=(f_1,\ldots, f_k)\colon M\to N$ there exists an index $i_0\in \{1,\ldots, k\}$ for which
\begin{enumerate}
\item the coordinate map $f_{i_0}\colon M\to N_{i_0}$ is a conformal map and
\item for $i\ne i_0$ the coordinate map $f_i \colon M\to N_i$ is constant.
\end{enumerate}
\end{named}

\begin{proof}
Let $F=(f_1,\ldots, f_k)\colon M\to N$  be a non-constant $\vol_N^\times$-calibrated curve. Since $F$ is a $(1+\delta)$-quasiregular $\vol_N^\times$-curve for every $\delta >0$, Theorem \ref{thm:stability-qr-curves} yields that, for every $\varepsilon >0$, there exists a unique index $i_\varepsilon \in \{1,\ldots,k\}$ for which the coordinate map $f_{i_\varepsilon}$ is a $(1+8k\sqrt{\varepsilon})$-quasiregular local homeomorphism. Each index $i_\varepsilon$ also satisfies the condition that $\norm{Df_i} \le 6\sqrt{k}\varepsilon^\frac{1}{4}(1+\varepsilon)^\frac{1}{n}\norm{DF}$ almost everywhere for $i\ne i_\varepsilon$.

By the uniqueness of the indices $i_\varepsilon$, there exists $i_0\in \{1,\ldots, k\}$ satisfying $i_0=i_\varepsilon$ for $\varepsilon >0$. For all $\varepsilon >0$, the coordinate map $f_{i_0}$ is a $(1+8k\sqrt{\varepsilon})$-quasiregular local homeomorphism and we have that $\norm{Df_i} \le 6\sqrt{k}\varepsilon^\frac{1}{4}(1+\varepsilon)^\frac{1}{n}\norm{DF}$ almost everywhere for $i\ne i_0$. This yields the claim.
\end{proof}

We finish by recalling the statement of Theorem \ref{thm:local-embedding} and giving its proof.

\begin{named}{Theorem \ref{thm:local-embedding}}
Let $n\ge 3$, $k\ge 1$ and $H>1$. Then there exists $\varepsilon=\varepsilon(n,k,H)>0$ having the property that each non-constant $(1+\varepsilon)$-quasiregular $\vol_N^\times$-curve $M\to N$ from an oriented and connected Riemannian $n$-manifold $M$ to a Riemannian product $N=N_1\times \cdots \times N_k$ of oriented Riemannian $n$-manifolds is a local $H$-quasiconformal embedding.
\end{named}

\begin{proof}
Let $H>H'>1$ and let $\varepsilon'=\varepsilon'(n,k,H')>0$ be as in Theorem \ref{thm:local-injectivity}. Let now $0<\varepsilon<\varepsilon'$.

Let $M$ be an oriented and connected Riemannian $n$-manifold and let $N=N_1\times \dots \times N_k$ be a Riemannian product of oriented Riemannian $n$-manifolds. Let also $F=(f_1,\ldots,f_k) \colon M\to N$ be a non-constant $(1+\varepsilon)$-quasiregular $\vol_N^\times$-curve.

The claim is local, so we may fix a point on $M$ and prove the claim in its neighborhood. Let $x\in M$ and let $a\in \N$ be an auxiliary parameter to be fixed later. Let also $(U,\varphi)$ and $(V,\psi)$ be charts as in the beginning of the section. Then the point $\varphi(x)$ has a neighborhood $\Omega \subset \varphi(U)$ for which the restriction $G_{|\Omega} \colon \Omega \to (\R^n)^k$, where $G=\psi \circ F\circ \varphi^{-1}$, is a non-constant $(1+\varepsilon')$-quasiregular $\vol_{(\R^n)^k}^\times$-curve.

By Theorem \ref{thm:local-injectivity}, the map $G_{|\Omega}$ is a local $H'$-quasiconformal embedding. This yields immediately that $F_{|\varphi^{-1}(\Omega)}$ is a local embedding. Further,
\begin{align*}
&\limsup_{r\to 0} \frac{\sup_{d_M(y,x)=r} d_N(F(y),F(x))}{\inf_{d_M(y,x)=r} d_N(F(y),F(x))} \\
&\quad \le (1+1/a)^2 \limsup_{r\to 0} \frac{\sup_{|z-\varphi(x)|=r} |G(z)-G(\varphi(x))|}{\inf_{|z-\varphi(x)|=r} |G(z)-G(\varphi(x))|} \le (1+1/a)^2 H'.
\end{align*}
Since we may assume that $(1+1/a)^2 H'\le H$, this concludes the proof.
\end{proof}




\bibliographystyle{abbrv}
\bibliography{QRC}

\end{document}